\documentclass[11pt,reqno]{amsart}

\usepackage{amsmath,amsthm,amssymb,amsfonts,enumerate,color,enumerate}
\newcommand{\mysection}[1]{\section{#1}
      \setcounter{equation}{0}}

\newcommand\cbrk{\text{$]$\kern-.15em$]$}} 
\newcommand\opar{\text{\raise.2ex\hbox{${\scriptstyle | }$}\kern-.34em$($} }
\newcommand\loc{\rm loc\,}
\newcommand\sca{\text{\sc a}}
\newcommand\scb{\text{\sc b}}
\newcommand\scc{\text{\sc c}}

\DeclareMathOperator{\sign}{sign}

\newtheorem{theorem}{Theorem}[section]
\newtheorem{lemma}[theorem]{Lemma}

\newtheorem{corollary}[theorem]{Corollary}

\theoremstyle{definition}
\newtheorem{assumption}{Assumption}[section]
\newtheorem{definition}{Definition}[section]
\newtheorem{example}{Example}[section]

\theoremstyle{remark}
\newtheorem{remark}{Remark}[section]

\newcommand\bG{\mathbb{G}}
\newcommand\bH{\mathbb{H}}

\newcommand\bK{\mathbb{K}}

\newcommand\bR{\mathbb{R}}

\newcommand\bB{\mathbb{B}}

\newcommand\bV{\mathbb{V}}

\newcommand\frf{\mathfrak{f}}

\newcommand\frh{\mathfrak{h}}

\newcommand\frp{\mathfrak{p}}

\newcommand\frJ{\mathfrak{J}}
\newcommand\frK{\mathfrak{K}}

\newcommand\cB{\mathcal{B}}

\newcommand\cF{\mathcal{F}}

\newcommand\cL{\mathcal{L}}
\newcommand\cM{\mathcal{M}}

\newcommand\cP{\mathcal{P}}

\newcommand\cS{\mathcal{S}}

\makeatletter
 \newcommand{\sumstar}%
 {\operatornamewithlimits{\sum@\kern-.2em\raise1ex\hbox{*}}}
 \makeatother

\def\Xint#1{\mathchoice
{\XXint\displaystyle\textstyle{#1}}%
{\XXint\textstyle\scriptstyle{#1}}%
{\XXint\scriptstyle\scriptscriptstyle{#1}}%
{\XXint\scriptscriptstyle\scriptscriptstyle{#1}}%
\!\int}
\def\XXint#1#2#3{{\setbox0=\hbox{$#1{#2#3}{\int}$ }
\vcenter{\hbox{$#2#3$ }}\kern-.6\wd0}}

\def\dashint{\Xint-}

\begin{document}
 
\title[Evolution equations with monotone operators]{Once again
on evolution equations with monotone operators
in Hilbert spaces and applications}

\author[I. Gy\"ongy]{Istv\'an Gy\"ongy}
\address{School of Mathematics and Maxwell Institute,
University of Edinburgh,
King's  Buildings,
Edinburgh, EH9 3JZ, United Kingdom}
\email{i.gyongy@ed.ac.uk}

\author[N.V. Krylov]{Nicolai V. Krylov}%
\thanks{}
\address{127 Vincent Hall, University of Minnesota,
Minneapolis,
       MN, 55455, USA}
\email{nkrylov@umn.edu}


\begin{abstract} 
Existence, uniqueness and stability of the solutions 
of linear stochastic evolution equations are investigated. The results 
obtained are used to prove theorems on solvability of linear second order 
stochastic partial differential equations
in  $L_{p}$-setting with singular lower order coefficients. 
\end{abstract}

\dedicatory{In memory of Giuseppe Da Prato}
\maketitle
\mysection{Introduction}

In this paper we present an extension of the classical 
theory of stochastic evolutional 
equations in order to cover a large class of linear second order 
stochastic partial differential 
equations (SPDEs) with singular lower order coefficients. 

We use the variational approach to prove existence, 
uniqueness and regularity of the solutions to 
stochastic evolution equations formulated in the framework 
of Hilbert spaces $V, H$, where $V$ is continuously and densely embedded 
in $H$. 
A key role in this approach is played by apriori estimates provided 
by the help of an It\^o formula for $|v_t|^2_H$, for the  
square of the norm in $H$ of solutions to the stochastic 
evolution equation. Therefore, after a preliminary section,  
first we present a theorem on It\^o's formula, 
Theorem \ref{theorem 4.9.3}, which looks like a simple 
version of well-known It\^o formulas for the squared norm 
of semimartingales in Hilbert spaces. The difference is 
that instead of the square integrability of 
$|f_t|_H$ in $t$ we assume that it is only integrable over $[0,T]$, 
 where $f=(f_t)_{t\in[0,T]}$ is an $H$-valued component 
 of the free term, see \eqref{8.9.5}.  
 However, this seemingly unimportant 
improvement is essential in our applications to 
stochastic evolution equations with monotone operators 
in Sections 4, 5 and 6, where we prove theorems on 
existence, uniqueness and on stability of the solutions,  
see Theorems \ref{theorem 4.18.1}, \ref{theorem 4.20.1}  
and \ref{theorem 4.22.1}.
In Sections 7, 8 and 9 we apply these theorems 
to parabolic SPDEs 
on the whole state space $\bR^d$. Under the strong 
parabolicity condition, Assumption \ref{assumption 1}, 
we prove existence uniqueness and regularity results for 
$W^i_2$-solutions (for $i=0,1$, see Definition \ref{definition L2}) 
to SPDEs with singular coefficients, and estimate 
these solutions 
also in $L_p$-spaces.  

The theory of linear second order SPDEs is well-developed 
when the coefficients in the equations are bounded, or locally 
bounded and satisfy some growth conditions. 
  See, for example, \cite{AM_03}, \cite{GK_92},  
\cite{LR_15} and the references therein. 
Stochastic heat equations with white-noise drift and with distributional 
drifts are studied in \cite{AV_20} and in \cite{ABLM_24}. Well-posedness 
of stochastic partial differential inclusions 
with singular drift in divergence form is 
investigated in \cite{MS_18}. 
An exposition of the regular 
and singular stochastic Allen-Cahn equations is given in the book 
\cite{B_22}.  
Optimal rate of convergence estimates for finite difference approximations 
of stochastic heat equations with locally unbounded drifts are obtained in 
\cite{BDG_23}.  Strong convergence of discretisations with parabolic rate 1 
is established for stochastic Allen-Cahn-type 
equations in \cite{GS_24}. 

There are well-known classical results on the solvability 
of deterministic elliptic and parabolic PDEs with locally unbounded lower order 
coefficients, see the monograph \cite{LSU1967}. Recently essential 
progress has been achieved for elliptic and parabolic PDEs 
in reducing the summability conditions 
on these coefficients, see \cite{K2023} 
and the references therein. 

The present paper is influenced by \cite{K2022b}, which contains apriori 
$L_p$-estimates for the kind of SPDEs we are interested in this paper.   
As far as we know our theorems are the first results 
on the solvability of SPDEs with singular lower order coefficients. 
Our interest in these equations is partially motivated by the recent progress 
in the theory of stochastic differential equations (SDEs) with singular drifts, 
  see, e.g., \cite{G_24} and the references therein. 
In a continuation of the present paper we want to investigate stochastic filtering  
problems for such SDEs by the help of SPDEs with singular coefficients. 

 We  finish the introduction with the stipulation
that we use the plain symbol $N$  
for various constants
which may change in every new appearance
and, if we use them in a proof of a statement,
then they are supposed to depend only 
on those parameters that are listed in the statement
unless explicitely indicated otherwise, like
$N=N(...)$, which means that $N$ depends
only on what is inside the parantheses.
Sometimes we use $N$ with indices, like $N_{2}$,
to facilitate keeping track of these particular
constants, they stay the same within the proof
where they appear, but may be different
in different proofs.

\mysection{A resolvent operator $R_{\lambda}:H\to V$
in Hilbert space
setting}
                                      \label{section 4.23.1}

In this section we collect some facts
(probably well-known) proved in 
\cite{Kr_13}.
Let $V$ and $H$ be two Hilbert spaces with 
scalar products and norms $(\cdot,\cdot)_{V}$, $  | \cdot  | _{V}$
and $(\cdot,\cdot)_{H}$, $  | \cdot  | _{H}$, respectively. 
Assume that
$V\subset H$, $V$ is dense in $H$ (in the metric of $H$),
and $  | u  | _{H}\leq   | u  | _{V}$ for any $u\in V$.

The  norm in $V$ is obviously equivalent to
$$
\big(\lambda  | u  | ^{2}_{H}+  | u  | ^{2}_{V}\big)^{1/2},
$$
where $\lambda\geq0$ is any fixed number.  
Then take an $f\in H$ and observe that the linear functional
$ (f,u)_{H}$ is bounded as a linear functional on $V$.
By Riesz's representation theorem there exists
a unique $v=:R_{\lambda}f\in V$
such that
$$
 (f,u)_{H}=\lambda (v,u)_{H}
+(v,u)_{V}
\quad\forall u\in V.
$$
 
\begin{theorem}
                                          \label{theorem 4.13.1}
                                          
(i) The operator $R_{\lambda}$ is a symmetric as an operator
acting in $H$ into $H$ and as an operator acting in $V$
into $V$, and for any $f\in H,u\in V$
\begin{equation}
                                               \label{7.27.3}
( R_{\lambda}f,u)_{V}=((1-\lambda R_{\lambda})f,u )_{H};
\end{equation}

(ii) The norms of the operator $ \lambda R_{\lambda}$ as an operator from 
$H$ into $H$
 as well as an operator from $V$ into $V$
 are less than or equal to one;

(iii) If $f\in H$, $\lambda\geq0$, and $\lambda R_{\lambda}
f=f$, then $f=0$;

(iv) The set $ R_{\lambda}H$ is dense in 
$V$ in the metric of $V$;

(v) For any $f\in H $ we have
\begin{equation*}
                                                   \label{4.9.5}
\lim_{\lambda\to\infty}  | f-\lambda R_{\lambda}f  | _{H}=0;
\end{equation*}

(vi) For
$f\in V$ we have
\begin{equation*}
                                                   \label{4.9.3}
\lim_{\lambda\to\infty}  | f-\lambda R_{\lambda}f  | _{V}=0.
\end{equation*}
\end{theorem}

\begin{remark}
                               \label{remark 4.19.1}
Since $|((1-\lambda R_{\lambda})f,u )_{H}|
\leq 2  | f  | _{H}  | u  | _{H}\leq2  | f  | _{H}  | u  | _{V}$
equation \eqref{7.27.3} implies that
$  | R_{\lambda}f  | _{V}\leq 2  | f  | _{H}$ for any $f\in H$.
\end{remark}

\mysection{It\^o's formula for the squared norm}
                                       \label{section 7.17.2}

Let $(\Omega,\cF,P)$ be a complete probability space
and let $\{\cF_{t},t\geq0\}$ be an increasing filtration
of $\sigma$-fields $\cF_{t}\subset \cF$, which are complete
with respect to $\cF,P$.   Let $\cP$ denote the 
predictable $\sigma$-field. 

In order to avoid unimportant complications
we assume that $(V,(\cdot,\cdot)_{V})$ is 
a separable Hilbert space, which is the case
in many applications. Then $(H,(\cdot,\cdot)_{H})$
is also separable. It is convenient that
under this assumption
there is no difference between weak and strong measurability.

Assume that we are given $V $-valued
processes $v_{t},v^{*}_{t}$, $t>0$, and an $H$-valued
$f_{t}$, $t>0$, 
 which are predictable and satisfy
\begin{equation}
                                                   \label{8.9.5}
E\int_{0}^{T}  | v_{t},v^{*}_{t}  | ^{2}_{V }
\,dt+E\Big(\int_{0}^{T}  | f_{t}  |  _{H }
\,dt\Big)^{2}<\infty
\end{equation}
  for any $T\in(0,\infty)$ (observe $  | f_{t}  |  _{H }$
  and not $  | f_{t}  | ^{2} _{H }$). Also let  $m_{t}$, $t\geq0$,
be an $H $-valued
continuous martingale starting at the origin with
\begin{equation*}
                                                     \label{8.9.1}
d\langle m\rangle_{t}\leq 
dt.
\end{equation*}

The theory of integrating predictable 
Hilbert-space valued processes with respect to
 continuous same space-valued martingales is quite
parallel to that in case the Hilbert space is just $\bR^{d}$.
This theory implies that,
for any predictable $H$-valued
process $h_{\cdot}\in L_{2}([0,\infty),H)$
the stochastic integral
\begin{equation*}
                                               \label{6.18.6}
M_{t}:=\int_{0}^{t}(h_{s},dm_{s})_{H },
\end{equation*}
is well defined and  is a
continuous real-valued martingale with
$$
\langle M\rangle_{t}\leq \int_{0}^{t}  | h_{s}  | ^{2}_{H }\,ds.
$${
Suppose that $v_{0}$ is
an 
$H $-valued $\cF_{0}$-measurable random vector
 and $\gamma_{t}\geq0$ is a predictable process
such that    
$$
E\int_{0}^{\infty}\gamma^{2}_{t}\,dt<\infty.
$$
Set
$$
m^{\gamma}_{t}=\int_{0}^{t}\gamma_{s}\,dm_{s}.
$$

Finally, assume that for any $v \in V$
we have
\begin{equation}
                             \label{4.9.7}
( v ,v_{t})_{H }=( v ,v_{0})_{H }+\int_{0}^{t}
[( v ,v^{*}_{s})_{V }+( v ,f_{s})_{H}]\,ds+
 ( v  ,m^{\gamma}_{t})_{H }
\end{equation}
for almost all $(\omega,t)$.   Here 
and later on ``almost all $(\omega,t)$" means $P\otimes dt$-almost 
all $(\omega,t)$.}

The following theorem looks very much like
what one can find in the literature including
\cite{Kr_13} except that in the condition \eqref{8.9.5}
the last square is inside the integral.
This seemingly unimportant improvement will
show up later in the existence theorem for SPDEs
with quite singular first order term.

\begin{theorem}
                                  \label{theorem 4.9.3}
Under the above assumptions there exists a continuous
$H $-valued $\cF_{t}$-adapted process $u_{t}$
and a set $\Omega'\subset\Omega$ of full probability
such that 

(i) $u_{t}=v_{t}$ for almost all $(\omega,t)$, so that
\begin{equation}
                                   \label{6,21.2}
E\int_{0}^{T}  | u_{t}   | ^{2}_{V }
\,dt<\infty
\end{equation}
  for any $T\in(0,\infty)$,

(ii) for
all $\omega\in\Omega'$, all $ v \in V$,
and all $t\geq0$ we have
\begin{equation} 
                                                 \label{4.9.8}
( v ,u_{t})_{H }=( v ,v_{0})_{H }+\int_{0}^{t}
[( v ,v^{*}_{s})_{V }+( v ,f_{s})_{H}]\,ds
 +( v ,m^{\gamma}_{t})_{H },
\end{equation}

(iii) for
all $\omega\in\Omega'$ 
and all $t\geq0$ we have 
\begin{equation}
                                                 \label{4.9.9}
  | u_{t}  | ^{2}_{H }=  | v_{0}  | ^{2}_{H }+2\int_{0}^{t}
[(u_{s},v^{*}_{s})_{V }+(u_{s},f_{s})_{H}]\,ds
+\langle m^{\gamma}\rangle_{t}
+2 \int_{0}^{t}(\gamma_{s}v_{s},dm_{s})_{H }.
\end{equation}

\end{theorem}

Proof. We closely follow the proof
of Theorem 2.1 in  \cite{Kr_13}. For $n=1,2,...$ define $S_{n}=nR_{n}$ and
\begin{equation*}
                                                    \label{8.6.1}
u^{n}_{t}=S_{n}v_{0}+\int_{0}^{t}[n (1-S_{n})  
v^{*}_{s}+S_{n}f_{s}]\,ds+ S_{n}m^{\gamma}_{t}.
\end{equation*}
Here the integral  makes sense as the integral  of an $H $-valued
function.
Furthermore, $u^{n}_{t}$ is obviously continuous
as an $H $-valued function.

Also observe that \eqref{4.9.7}  with $ v =S_{n}\psi$,
$\psi\in H $, and \eqref{7.27.3} yield that for almost all
$(\omega,t)$
$$
(\psi,S_{n}v_{t})_{H }=(\psi,S_{n}v_{0})_{H }
+\int_{0}^{t}(\psi, n(1-S_{n}) 
v^{*}_{s}+S_{n}f_{s})_{H }\,ds
$$
\begin{equation*}
                                            \label{8.6.2}
+(\psi, S_{n}m^{\gamma}_{t})_{H }=
(\psi,u^{n}_{t})_{H }.
\end{equation*}
This and the separability of $H $ shows that 
$$
u^{n}_{t}=S_{n}v_{t}
$$ 
for almost all $(\omega,t)$.

 Next, from Doob's inequality
it follows that for any $T\in[0,\infty)$
$$
 E\sup_{t\leq T}  | u^{n}_{t}   | ^{2}_{H }<\infty.
$$

By It\^o's formula for integrals of Hilbert-space 
valued processes
(see, for instance Theorem 4.32 of \cite{DZ_14}) 
we have (a.s.)
$$
  | u^{n}_{t}  | _{H }^{2}=  | S_{n}v_{0}  | _{H }^{2}
+2\int_{0}^{t}(u^{n}_{s}, n(1-S_{n}) 
v^{*}_{s}+S_{n}f_{s})_{H } \,ds
$$
\begin{equation}
                                                    \label{4.11.1}
+\langle S_{n}m^{\gamma}\rangle_{t}
+2\int_{0}^{t}(S_{n}u^{n}_{s}, dm^{\gamma}_{s})_{H },
\end{equation}
$$
  | u^{n}_{t}-u^{k}_{t}  | ^{2}_{H }=  | (S_{n}-S_{k})v_{0}  | _{H }^{2}
+\langle (S_{n}-S_{k})m^{g}\rangle_{t}
+2\int_{0}^{t}(S_{n}u^{n}_{s}-S_{k}u^{k}_{s}, dm^{\gamma}_{s})_{H }
$$
\begin{equation}
                                                    \label{4.11.2}
+2\int_{0}^{t}(u^{n}_{s}-u^{k}_{s},
 [n(1-S_{n}) -k(1-S_{k})  ]
v^{*}_{s}+(S_{n}-S_{k})f_{s})_{H } \,ds
\end{equation}
for all $t\geq0$.

Observe that there is $\Omega'$ with $P(\Omega')=1$ such that
\begin{equation}
                                                 \label{6.19.5}
u^{n}_{t}=S_{n}v_{t},\quad n=1,2,...,\quad
 v_{t},v^{*}_{t}\in V 
\end{equation}
 for almost all $t$ on $\Omega'$. It follows that
in the integrands in \eqref{4.11.1} and \eqref{4.11.2}
 we can replace
$u^{n}_{s}$ with $S_{n}v_{s}$ if $\omega\in\Omega'$
and use \eqref{7.27.3}. Then
for $\omega\in\Omega'$ and $t$ such that \eqref{6.19.5} holds
we have
$$
(u^{n}_{s}-u^{k}_{s},
 [n(1-S_{n}) -k(1-S_{k}) ]
v^{*}_{s})_{H }
$$
$$
=(S_{n}v_{s}-S_{k}v_{s},
 n(1-S_{n}) v^{*}_{s})_{H }
-(S_{n}v_{s}-S_{k}v_{s},k(1-S_{k}) 
v^{*}_{s})_{H }
$$
$$
=(S_{n}v_{s}-S_{k}v_{s},S_{n}v^{*}_{s})_{V }-
(S_{n}v_{s}-S_{k}v_{s},S_{k}v^{*}_{s})_{V }
$$
$$
=(S_{n}v_{s}-S_{k}v_{s},S_{n}v^{*}_{s}-
S_{k}v^{*}_{s})_{V }.
$$

Hence, for $\omega\in\Omega'$ and all $n,k\geq1$ and
 $t\geq0$ we get that
$$
  | u^{n}_{t}  | _{H }^{2}=  | S_{n}v_{0}  | _{H }^{2}
+2\int_{0}^{t}(S_{n}v_{s},
v^{*}_{s})_{V } \,ds
$$
$$
+2\int_{0}^{t}(u^{n}_{s},S_{n}f_{s})_{H}\,ds
+\langle S_{n}m^{\gamma}\rangle_{t}
+2\int_{0}^{t}(S_{n}u^{n}_{s}, dm^{\gamma}_{s})_{H },
$$
$$
  | u^{n}_{t}-u^{k}_{t}  | ^{2}_{H }=  | (S_{n}-S_{k})v_{0}  | _{H }^{2}
+2\int_{0}^{t}(S_{n}v_{s}-S_{k}v_{s},S_{n}v^{*}_{s}-
S_{k}v^{*}_{s})_{V } \,ds
$$
$$
+\langle (S_{n}-S_{k})m^{\gamma}\rangle_{t}
+2\int_{0}^{t}(S^{2}_{n}v_{s}-S^{2}_{k}v_{s},
 dm^{\gamma}_{s})_{H }
$$
$$
+2\int_{0}^{t}(u^{n}_{s}-u^{k}_{s},(S_{n}-S_{k})f_{s})_{H}\,ds.
$$

Furthermore, by Doob's inequality for any $T\in[0,\infty)$
\begin{equation}
                                                    \label{4.11.4}
E\sup_{t\leq T}  | u^{n}_{t}-u^{k}_{t}  | ^{2}_{H }
\leq I_{nk}^{1}+2I^{2}_{nk}+I^{3}_{nk}+4(I^{4}_{nk})^{1/2}
+2I^{5}_{mk},
\end{equation}
where
$$
I_{nk}^{1}=E  | (S_{n}-S_{k})v_{0}  | _{H }^{2}
$$
$$
I_{nk}^{2}= E\int_{0}^{T}|((S_{n}-S_{k})v_{s},
(S_{n}-S_{k})
v^{*}_{s})_{V} |\,ds,
$$
$$
I_{nk}^{3}=E
\langle (S_{n}-S_{k})m^{\gamma}\rangle_{T}
=E  | (S_{n}-S_{k})m^{\gamma}_{T} | ^{2}
_{H },
$$
$$
I_{nk}^{4}=E \int_{0}^{T}   | S^{2}_{n}v_{s}
-S^{2}_{k}v_{s}  | _{H }^{2}\,ds ,
$$
$$
I^{5}_{mk}=E\int_{0}^{T}
  | u^{n}_{s}-u^{k}_{s}  | _{H}  | (S_{n}-S_{k})f_{s}  | _{H}\,ds.
$$
By using the dominated convergence theorem, 
Theorem \ref{theorem 4.13.1}, and the inequality
$$
|((S_{n}-S_{k})v_{s}, (S_{n}-S_{k})
v^{*}_{s})_{V } |
$$
$$
\leq  | (S_{n}-S_{k})v_{s}  | ^{2}_{V }
+  | (S_{n}-S_{k})v^{*}_{s}  | ^{2}_{V },
$$
we easily conclude that $I_{nk}^{1}+2I^{2}_{nk}+I^{3}_{nk}\to0$
as $n,k\to\infty$. Furthermore, $S_{n}^{2}-S_{k}^{2}
=(S_{n}+S_{k})(S_{n}-S_{k})$ so that
$$
I^{4}_{nk}\leq 4 E\int_{0}^{T}  | (S_{n}-S_{k})
v_{s}  | _{H }^{2} \,ds, 
$$
which by the dominated convergence theorem
implies that $I^{4}_{nk}\to0$ 
as $n,k\to\infty$ as well. Finally,
$$
I^{5}_{nk}\leq (1/4)E\sup_{t\leq T}
  | u^{n}_{t}-u^{k}_{t}  | ^{2}_{H }
+4E\Big(\int_{0}^{T}  | (S_{n}-S_{k})f_{s}  | _{H}\,ds\Big)^{2}.
$$

 We now conclude from \eqref{4.11.4} that its left-hand side
tends to zero. Furthermore,
$$
E\int_{0}^{T}  | u^{n}_{t}-v_{t}  | ^{2}_{V }\,dt
=E\int_{0}^{T}  | (S_{n} -1)v_{t}  | ^{2}_{V }\,dt\to0.
$$
Hence $u^{n}_{t}$ converges to $v_{t}$ in $L_{2}
(\Omega\times(0,T),V)$
and converges uniformly on $[0,T]$ as $H $-valued functions
in probability. The latter limit we denote by $u_{t}$
and show that this function is the one we want.
Of course, $u_{t}$ is a continuous $H $-valued functions,
 it is $\cF_{t}$-adapted, and $u_{t}=v_{t}$ for almost all
$(\omega,t)$.

One easily obtains that for each $t$ 
equation \eqref{4.9.9} holds with probability
one by passing to the limit in \eqref{4.11.1}.
Since both parts of \eqref{4.9.9} are continuous in $t$,
it holds on a set of full probability for all $t$.

Obviously  \eqref{4.9.7} will hold for almost all $(\omega,t)$
if we replace $v_{t}$ with $u_{t}$, that is,  \eqref{4.9.8}
holds for any $ v \in V$ for almost all
$(\omega,t)$. The continuity of both parts of 
\eqref{4.9.8} with respect to $t$ and $ v \in V$
and the separability of $V$ then imply that there is
a set $\Omega'$ of full probability such that assertion (iii)
holds.

The theorem is proved.
\begin{remark}
The reader understands, of course, that condition
\eqref{8.9.5} can be  replaced with the same condition but
without expectation sign 
and with $T$ replaced by any stopping time
$\tau$. Of course, then the same changes should
be applied to \eqref{6,21.2}. This generalization
is easily achieved by using appropriate stopping times.
\end{remark}

  \mysection{Linear equations with  monotone
  operators}
                                  \label{section 4.18.1}
Traditionally, equations wih monotone operators
involve $V$ and its conjugate $V^{*}$.                                 
The duality between $v\in V$ and $v^{*}\in V^{*}$
is denoted by $\langle v, v^{*}\rangle$.
  
Fix $T\in(0,\infty)$ and assume for any 
$ t\in[0,T],\omega\in\Omega$ we are given
linear operators $A_{t}:V\to V^{*}$, $B_{t}:V\to \ell_{2}(H)$,
such that, for each $v\in V$, $A_{t}v,B_{t}v$ are
Lebesgue measurable
with respect to $(t,\omega)$ and are $\cF_{t}$-adapted.

We single out the following coercivity and boundedness
assumptions.

\begin{assumption}
                                               \label{assumption 4.18.1}
There are constants $\delta>0, K<\infty$ 
 and $K_0<\infty$  
such that, for any
$v\in V, t\in[0,T],\omega\in\Omega$ we have
\begin{equation}
                                    \label{4.18.2}
2\langle v, A_{t}v\rangle+\sum_{k}  | B^{k}_{t}v  | ^{2}_{H}
\leq -\delta  | v  | ^{2}_{V}\,,
\end{equation}
\begin{equation}
                                    \label{4.18.30}
  | A_{t}v  | _{V^{*}}\leq K  | v  | _{V}
   ,
  \quad 
  \sum_{k}  | B^{k}_{t}v  | ^{2}_{H}\leq K_0^2|v|^2_{V} .
\end{equation}

\end{assumption}

\begin{remark}
                              \label{remark 4.19.2} 
Observe that
$$
\sum_{k}  | B^{k}_{t}v  | ^{2}_{H}\leq-2\langle v, A_{t}v\rangle
\leq 2K  | v  | _{V}^{2},
$$
 i.e., the ``coercivity" condition \eqref{4.18.2}
implies an upper bound for the operator norm of $B$ in terms of an 
upper bound for the operator norm for $A$. 
Nevertheless, we introduce also 
an upper bound $K_0$ for the operator norm of $B_t$, 
because that is 
what plays a role in the estimate 
\eqref{4.27.1} below. 
 
\end{remark}

\begin{definition}
                                    \label{definition 4.20.1}
Denote by $\bV$ the Banach space of $H$-valued,
continuous, and $\cF_{t}$-adapted functions
$u_{t}$ on $[0,T]\times\Omega$ such that $u_{t}\in V$
for almost all $(t,\omega)$ and
\begin{equation*}
                                      \label{4.18.6}
  | u_{\cdot}  | _{\bV}^{2}:=E\sup_{t\leq T}  | u_{t}  | _{H}^{2}
+E\int_{0}^{T}  | u_{t}  | ^{2}_{V}\,dt<\infty.
\end{equation*}
\end{definition}

For given $u_{0}$, we will be dealing with the problem
\begin{equation}
                            \label{4.24.1}
du_{t}=\big[ A_{t}u_{t}+f^{*}_{t} 
+ f_{t}+g_{t} \big]\,dt
+\big( B^{k}_{t}u_{t} +h^{k}_{t}\big) \,dw^{k}_{t},\quad
t\leq T,\quad u_{t}\big|_{t=0}=u_{0}.
\end{equation}

\begin{definition}
                             \label{definition 4.25.1}
An $H$-valued, $\cF_{t}$-adapted,
$H$-continuous in $t$ function $u_{t}$ on $\Omega\times[0,T]$
is called an  
 $H$-solution of \eqref{4.24.1} if
$u_{\cdot}\in V$ for almost all $(\omega,t)$
$$
\int_{0}^{T}  | u_{t}  | _{V}^{2}\,dt<\infty
$$
(a.s.) and for any $ v \in V$, (a.s.) for all $t\in[0,T]$,
$$
( v ,u_{t})_{H}=( v ,u_{0})_{H}+\sum_{k}\int_{0}^{t}
\big( v ,B^{k}_{s}u_{s}+h^{k}_{s}\big)_{H}\,dw^{k}_{s}
$$
\begin{equation}
                                      \label{4.18.5}
+\int_{0}^{t}\big[\langle v ,A_{s}u_{s}+f^{*}_{s}\rangle
+( v ,f_{s}+g_{s})_{H}\big]\,ds.
\end{equation}
                             
\end{definition}
 
\begin{theorem}
                               \label{theorem 4.18.1}
Let Assumption \ref{assumption 4.18.1} hold and let $\varepsilon>0$.                              
Assume that on $\Omega\times[0,T]$
we are given a $V^{*}$-valued function $f^{*}_{t}$,
   $H$-valued functions $f_{t},g_{t}$,
and an $\ell_{2}(H)$-valued $h_{t}=(h^{k}_{t},k=1,2,...)$,
which are Lebesgue measurable in $(t,\omega)$,  
$\cF_{t}$-adapted, and such that
\begin{equation*}
                                \label{4.18.030}
E\Big(\int_{0}^{T}  |g_{t} |_{H}\,dt\Big)^{2}
+E\int_{0}^{T}\Big( |f^{*}_{t} |^{2}_{V^{*}}
+\alpha^{-1}_{t}|f_{t}|^{2}_{H}+
 |h^{\cdot} _{t} |^{2}_{\ell_{2}(H)}\Big)\,dt<\infty,
\end{equation*}
where $\alpha_{t}\geq\varepsilon $ is a   predictable 
function such that
$$
\sup_{\Omega}\int_{0}^{T}\alpha_{t}\,dt<\infty.
$$

Then for any $H$-valued $\cF_{0}$-measurable $u_{0}$
such that $E  | u_{0}  | ^{2}_{H}<\infty$,

(i) there exists
an   $H$-solution of \eqref{4.24.1}.
For this solution we have  $u_{\cdot}\in\bV$. 

(ii) For this solution we have
 $$
 E \sup_{t\leq T}|u_{t}|_{H}^{2}e^{-2  \phi_{t}}
 + E\int_{0}^{T}e^{-2  \phi_{s}}|u_{s}
 |_{V}^{2}\,ds+  E\int_{0}^{T}\alpha_{s}
 |u_{s}|_{H}^{2}e^{-2  \phi_{s}}\,ds
 $$
 $$
 \leq  N E|u_0|^2_H 
 +
 NE\int_{0}^{T}
 e^{-2  \phi_{s}}\big(|f^{*}_{s}
 |_{V^{*}}^{2}+\alpha_{s}^{-1}|f_{s}|_{H}^{2}
 +|h^{\cdot}_{s}|^{2}_{\ell_{2}(H)}\big)\,ds
 $$
\begin{equation}
                                                \label{4.27.1}
+N E\Big(\int_{0}^{T}e^{-   \phi_{s}}
 |g_{s}|_{H}\,ds\Big)^{2},
\end{equation}
where (and in the proof) 
$$
\phi_{t}:=\int_{0}^{t}\alpha_{s}\,ds.
$$
and the 
(finite) constants $N$ depend
only on $\delta,K_{ 0}$.

Moreover, if we have another $H$-solution $u'_{t}$, then 
$$
P(\sup_{t\leq T}  | u_{t}-u'_{t}  | _{H}=0)=1.
$$
\end{theorem}

\begin{remark}
                                     \label{remark 4.20.1}
The uniqueness statement in Theorem \ref{theorem 4.18.1}
follows immediately from Theorem \ref{theorem 4.9.3}. Indeed, for
$v_{t}:=u_{t}-u'_{t}$. we have 
$$
  | v_{t}  | ^{2}_{H}\leq 2\int_{0}^{t}(v_{s},dm_{s})_{H},
$$
where $m_{s}$ is certain $H$-valued martingale. Hence
the stochastic integral above is nonnegative implying that
it is zero and $E\sup_{t\leq T}  | v_{t}  | _{H}=0$.
\end{remark}

\begin{lemma}
                                     \label{lemma 4.19.1}
Assertion  (i) of Theorem \ref{theorem 4.18.1}
is true under the additional assumption that
$f^{*}=0,B=0,h=0$.

\end{lemma}

Proof. Let $Q:V^{*}\to V$ be defined as the natural
identification of $V^{*}$ with $V$ (so that $\langle v,
v^{*}\rangle=(v,Qv^{*})_{V}$). Define $A^{n}_{s}=
n(1-S_{n})QA_{s}S_{n}$ and
in $H$ for $n\geq1, t\leq T$ consider the following equation
\begin{equation}
                                                            \label{4.19.1}
u^{n}_{t}=S_{n}u_{0}+\int_{0}^{t}\big[A^{n}_{s}
u^{n}_{s}+ F_{s}\big]\,ds,\quad F_{s}=f_{s}+g_{s}.
\end{equation}
Here
$$
E\Big(\int_{0}^{T}|f_{t}|_{H}\,dt\Big)^{2}
\leq E \int_{0}^{T}\alpha_{t}\,dt
\int_{0}^{T}\alpha_{t}^{-1}|f_{t}|_{H}^{2}\,dt<\infty,
$$
so that $F_{\cdot}\in L_{1}([0,T],H)$ (a.s.).

Also observe that  for any
$u\in H$, $s\leq T$
$$
  | A^{n}_{s}u  | _{H}\leq 2n  | QA_{s}S_{n}u  | _{H}
\leq 2n  | QA_{s}S_{n}u  | _{V}=2n  | A_{s}S_{n}u  | _{V^{*}}
$$
$$
\leq 2nK  |  S_{n}u  | _{V }\leq 4n^{2}K  | u  | _{H }.
$$
It follows that, for each $\omega$, equation
\eqref{4.19.1} is a first-order linear differential
equation in $H$ with bounded opertors $A^{n}_{s}$.
Therefore, it has a unique solution, which can be found
by successive approximations  showing that the solution
$u^{n}_{t}$ inherits the measurability properties of $F_{t}$
and is a continuous   $H$-valued function.

Furthermore
$$
  | u^{n}_{t}  | ^{2}_{H}=  | S_{n}u_{0}  | ^{2}_{H}
+2\int_{0}^{t}\big(u^{n}_{s},A^{n}_{s}u^{n}_{s}+ F_{s}\big)_{H}\,ds,
$$
where
\begin{equation}
                                       \label{4.19.5}
2\big(u^{n}_{s},A^{n}_{s}u^{n}_{s}\big)_{H}=
2\big(S_{n}u^{n}_{s},QA _{s}S_{n}u^{n}_{s}\big)_{V}
=2\langle S_{n}u^{n}_{s}, A _{s}S_{n}u^{n}_{s}\rangle
\leq-\delta  | S_{n}u^{n}_{s}  | _{V}^{2}.
\end{equation}

Consequently,
$$
\sup_{t\leq T}  | u^{n}_{t}  | ^{2}_{H}+ \delta
\int_{0}^{T}  | S_{n}u^{n}_{s}  | _{V}^{2}\,ds
\leq   | S_{n}u_{0}  | ^{2}_{H}+2\sup_{t\leq T}  | u^{n}_{t}  |  _{H}
\int_{0}^{T}  | F_{s}  | _{H}\,ds
$$
$$
\leq   | S_{n}u_{0}  | ^{2}_{H}+(1/2)\sup_{t\leq T}  | u^{n}_{t}  | ^{2} _{H}
+2\Big(\int_{0}^{T}  | F_{s}  | _{H}\,ds\Big)^{2},
$$
\begin{equation}
                                      \label{4.19.4}
\sup_{t\leq T}  | u^{n}_{t}  | ^{2}_{H}+2\delta
\int_{0}^{T}  | S_{n}u^{n}_{s}  | _{V}^{2}\,ds
\leq 2  | S_{n}u_{0}  | ^{2}_{H}+ 
4\Big(\int_{0}^{T}  | F_{s}  | _{H}\,ds\Big)^{2}.
\end{equation}

This estimate implies that there is a subsequence
$n'\to\infty$ such that the functions $S_{n'}u^{n'}_{\cdot}$
converge weakly in $L_{2}( \Omega\times[0,T],V)$
to a function $u_{\cdot}\in L_{2}( \Omega\times[0,T],V)$.
Also note that, for any $v\in V$, the linear operators
$$
F:L_{2}( \Omega\times[0,T],V)\to L_{2}( \Omega\times[0,T]),\quad
F(v_{\cdot})_{t}=(v,v_{t})_{H},
$$
$$
G:L_{2}( \Omega\times[0,T],V)\to L_{2}( \Omega\times[0,T]),\quad
G(v_{\cdot})_{t}=
\int_{0}^{t}\langle v,A_{s}v_{s}\rangle\,ds
$$
are bounded and hence weakly continuous. It follows that,
since,  for any $v\in V$,
$$
(v,S_{n'}u^{n'}_{t})_{H}=(S_{n'}v,S_{n'}u_{0})_{H}+
\int_{0}^{t}\big[\langle S^{2}_{n'}v,A_{s}S_{n'}u^{n'}_{s}\rangle
+(v,S_{n'}F_{s})_{H}\big]\,ds,
$$
for almost all $(t,\omega)\in  \Omega\times[0,T]$ we have
$$
(v,u_{t})=(v,u_{0})_{H}+\int_{0}^{t}
\big[\langle v,A_{s}u_{s}\rangle+(v,F_{s})_{H}\big]\,ds.
$$
After that it only remains to apply Theorem
\ref{theorem 4.9.3} (with $m_{t}=0$) and perform
some manipulations as after \eqref{4.19.5}. 
The    lemma  is proved.

{\bf Proof of Theorem \ref{theorem 4.18.1}}.
Denote by $u^{0}_{t}$   the $H$-solution from Lemma 
\ref{lemma 4.19.1} and observe that by setting
$v_{t}=u_{t}-u^{0}_{t}$ we can rewrite \eqref{4.18.5}
as
\begin{equation}
                                            \label{4.19.6}
( v ,v_{t})_{H}=\sum_{k}\int_{0}^{t}
\big( v ,B^{k}_{s}v_{s}+\bar h^{k}_{s}\big)_{H}\,dw^{k}_{s}
+\int_{0}^{t} \langle v ,A_{s}v_{s}+\bar f^{*}_{s}\rangle \,ds,
\end{equation}
where $\bar h^{k}_{s}=h^{k}_{s}+B^{k}_{s}u^{0}_{s}$,
$\bar f^{*}_{s}=f^{*}_{s}+A_{s}u^{0}_{s}$. Observe that
$$
E\int_{0}^{T}\sum_{k}  | B^{k}_{s}u^{0}_{s}  | _{H}^{2}\,ds
\leq  K_0^2 E\int_{0}^{T}  | u^{0}_{s}  | _{V}^{2}\,ds<\infty,
$$
$$
E\int_{0}^{T}  | A_{s}u^{0}_{s}  | ^{2}_{V^{*}}
\leq  K^{2}E\int_{0}^{T}  | u^{0}_{s}  | _{V}^{2}\,ds<\infty.
$$
After that the results in \S 2  of \cite{KR_79} allow us
to assert that equation \eqref{4.19.6} has a unique 
continuous $H$-valued 
solution $v_{t}$ such that $v_{\cdot}\in\bV$.
Upon setting $u_{t}=v_{t}+u^{0}_{t}$ we obtain the existence part
in Theorem \ref{theorem 4.18.1}. 

To prove \eqref{4.27.1} 
use Theorem \ref{theorem 4.9.3}
with
$$
 \sum_{k}\int_{0}^{t}\gamma^{-1}_{s}[
B^{k}_{s}u_{s}+h^{k}_{s}]\,dw^{k}_{s},
$$
in place of $m_{t}$, where
$$
\gamma^{2}_{s}=\sum_{k}|B^{k}_{s}u_{s}+h^{k}_{s}|^{2}_{H}
 ,\quad (0/0:=0) .
$$
Then we get  
 $$
 |u_{t}|_{H}^{2}e^{-2  \phi_{t}}=|u_{0}|^{2}_{H}
 -2\int_{0}^{t}\alpha_{s}
 |u_{s}|_{H}^{2}e^{-2  \phi_{s}}\,ds
 $$
 $$
+\int_{0}^{t}e^{-2  \phi_{s}}
 \big[2\langle u_{s}, A_{s}u_{s}+f^{*}_{s}\rangle
 +2(u_{s},f_{s}+g_{s})_{H}\big]\,ds
 $$
 \begin{equation*}
                                \label{4.28.2}
  +\sum_{k}\int_{0}^{t}e^{-2  \phi_{s}}
 |B^{k}_{s}u_{s}+h^{k}_{s}|^{2}_{H}\,ds+m_{t},
 \end{equation*}
 where
 $$
 m_{t}=  2\sum_{k}\int_{0}^{t}e^{-2  \phi_{s}}
 (u_{s},B^{k}_{s}u_{s}+h^{k}_{s})_{H}\,dw^{k}_{s}.
 $$
 
 Observe that
 $$
 2\langle u_{s}, A_{s}u_{s} \rangle
 +\sum_{k}|B^{k}u_{s}|^{2}
 _{H}\leq -\delta |u_{s}|^{2}_{V},
 $$
 $$
 2\langle u_{s},  f^{*}_{s}\rangle\leq(\delta/4)|u|^{2}_{V}
 +(4/\delta)|f^{*}_{s}|^{2}_{V^{*}},
 $$
 $$
 2(u_{s},f_{s})_{H}\leq\alpha_{s}|u_{s}|_{H}^{2}
 +\alpha^{-1}_{s}|f_{s}|^{2}_{H},
 $$
  $$
 \sum_{k}|B^{k}_{s}u_{s}+h^{k}_{s}|^{2}_{H}-\sum_{k}
 |B^{k}_{s}u_{s}|^{2}_{H}\leq
  \sum_{k}2|B^{k}_{s}u_{s}|_{H}
 |h^{k}_{s}|_{H}+ |h ^{\cdot}_{s}|_{\ell_{2}(H)}^{2}
 $$
 $$
 \leq  (\delta/4)|u_{s}|^{2}_{V}
 +[(  4K_0^2 /\delta)+1]|h ^{\cdot}_{s}|_{\ell_{2}(H)}^{2}.
 $$
 Hence,
$$
   |u_{t}|_{H}^{2}e^{-2  \phi_{t}}
 +(\delta/2)\int_{0}^{t}e^{-2  \phi_{s}}|u_{s}
 |_{V}^{2}\,ds +  \int_{0}^{t}\alpha_{s}
 |u_{s}|_{H}^{2}e^{-2  \phi_{s}}\,ds $$
 $$
 \leq  \int_{0}^{t}
 e^{-2  \phi_{s}}\big((4/\delta)|f^{*}_{s}
 |_{V^{*}}^{2}+\alpha^{-1}_{s}|f_{s}|^{2}_{H}\big)\,ds+
 2 \int_{0}^{t}e^{-2  \phi_{s}}
 |(u_{s},g_{s})_{H}|\,ds
 $$
\begin{equation}
                                                      \label{6,21.4}
  +[(  4K_0^2 /\delta)+1]  \int_{0}^{t}e^{- 2 \phi_{s}}
|h ^{\cdot}_{s}|_{\ell_{2}(H)}^{2}\,ds +m_{t}
=:I_{t}+m_{t}
\end{equation}
and 
$$
 J:=  \sup_{t\leq T}|u_{t}|_{H}^{2}e^{-2  \phi_{t}}
 +(\delta/2)\int_{0}^{T}e^{-2  \phi_{s}}|u_{s}
 |_{V}^{2}\,ds +  \int_{0}^{T}\alpha_{s}
 |u_{s}|_{H}^{2}e^{-2  \phi_{s}}\,ds 
 $$
 \begin{equation*}
                                     \label{4.27.5}
 \leq 2I_{T}+2\sup_{t\leq T}|m_{t}|.
\end{equation*}

By the Davis inequality  
 $$
 E\sup_{t\leq T}|m_{t}|\leq
 8E\Big(\int_{0}^{T}e^{-4 \phi_{s}}\sum_{k}
 (u_{s},B^{k}_{s}u_{s}+h^{k}_{s})^{2}_{H}\,ds  \Big)^{1/2}
 $$
 $$
 \leq 8E\sup_{s\leq T}e^{-  \phi_{s}}|u_{s}|_{H}
 \Big(\int_{0}^{T}e^{-2 \phi_{s}}\big(
   2K_0^2 |u_{s}|_{V}^{2}+2|h^{\cdot}_{s}|^{2}_{\ell_{2}(H)}
 \big)\,ds\Big)^{1/2}
 $$
 $$
 \leq (1/8)E\sup_{s\leq T}e^{- 2 \phi_{s}}|u_{s}|^{2}_{H}
 +NE\int_{0}^{T}e^{-2  \phi_{s}}|h
 ^{\cdot}_{s}|^{2}_{\ell_{2}(H)}\,ds+N_{1}E\int_{0}^{T}e^{-2  \phi_{s}}|u_{s}
 |_{V}^{2}\,ds,
 $$
 where the last term we estimate by taking expectations
 in \eqref{6,21.4}. Then we see that
 $$
 2N_{1}E\int_{0}^{T}e^{-2  \phi_{s}}|u_{s}
 |_{V}^{2}\,ds\leq (2/\delta)N_{1}EI_{T}.
 $$
 $$
 EJ\leq (1/4)E\sup_{s\leq T}e^{- 2 \phi_{s}}|u_{s}|^{2}_{H}
 +NE\int_{0}^{T}e^{- 2 \phi_{s}}(|f^{*}_{s}|^{2}_{V^{*}}+\alpha^{-1}_{s}
 |f_{s}|^{2}_{H}+|h
 ^{\cdot}_{s}|^{2}_{\ell_{2}(H)})\,ds
 $$
 $$
 +N_{2}E\int_{0}^{T}e^{-2  \phi_{s}}
 |(u_{s},g_{s})_{H}|\,ds,
 $$
 where the last term is dominated by
$$
N_{2}\Big(E\sup_{s\leq T}e^{- 2 \phi_{s}}|u_{s}|^{2}_{H}\Big)^{1/2}
 \Big(E\Big(\int_{0}^{T}e^{-   \phi_{s}}
 |g_{s}|_{H}\,ds\Big)^{2}\Big)^{1/2}
$$
$$
\leq (1/4)E\sup_{s\leq T}e^{- 2 \phi_{s}}|u_{s}|^{2}_{H}
+NE\Big(\int_{0}^{T}e^{-   \phi_{s}}
 |g_{s}|_{H}\,ds\Big)^{2}.
$$
As a result we have
$$
 EJ\leq (1/2)E\sup_{s\leq T}e^{- 2 \phi_{s}}|u_{s}|^{2}_{H}
 +NE\int_{0}^{T}e^{- 2 \phi_{s}}(|f^{*}_{s}|^{2}_{V^{*}}+\alpha^{-1}_{s}
 |f_{s}|^{2}_{H}+|h
 ^{\cdot}_{s}|^{2}_{\ell_{2}(H)})\,ds
 $$
 $$
+NE\Big(\int_{0}^{T}e^{-   \phi_{s}}
 |g_{s}|_{H}\,ds\Big)^{2},
 $$
 and this yields  \eqref{4.27.1}. The theorem is proved.

 \begin{remark} 
                                     \label{remark 1.10.5.2024}
 Let Assumption \ref{assumption 4.18.1} hold with 
 $$
 \big(\sum_k|B^k_tu|^2_{H}\big)^{1/2}\leq K_0|u|_{V}+K_1|u|_H, 
 $$
for a constant $K_1$, in place of the the second inequality in \eqref{4.18.30}.
Then by inspecting the proof of
estimate 
\eqref{4.27.1} it is easy to see that the estimate
 remains valid with a constant $N=N(\delta, K_0)$ if 
in addition to the assumption on $\alpha$ we suppose that
 $\alpha\geq K_1^2$.                                           
\end{remark}                                                        
  
\mysection{Generalization of Theorem 
\protect\ref{theorem 4.18.1}}
                                  \label{section 4.20.1}

Here, in addition to the assumptions stated
in the beginning of Section \ref{section 4.18.1} and
 in Theorem \ref{theorem 4.18.1},
we suppose that, on $[0,T]\times\Omega$, we are also given
 predictable functions $\sca_{t}$ 
with values in the set $\sca=\cL(V,H)$ of linear operators 
from $V$ to $H$, $\sca^{\ast}_{t}$
with values in the set $\sca^{\ast}=\cL(H,V^{\ast})$ 
of linear operators mapping $H$ into $V^{\ast}$, 
 $\scc_{t}$ with values in the set $\scc=\cL(H,H)$ 
of linear operators on $H$,  
and $\scb_{t}$ with values in the set
$\scb=\cL(H,\ell_{2}(H))$ of linear operators 
from $V$ to $\ell_{2}(H)$
such that, for some {\em constants\/} $|\sca|,|\sca^{*}|,|\scb| 
 ,  |\scc| <\infty$, 
$$
  \int_{0}^{T} |\sca_{t} |^{2}_{\sca}\,dt 
  \leq  |\sca |^{2},
  \quad
  \int_{0}^{T} 
  |\sca^{\ast}_{t} |^{2}_{\sca^{\ast}}\,dt\leq  |\sca^{\ast} |^{2},
  \quad
   \int_{0}^{T} |\scb_{t} |^{2}_{\scb}\,dt 
  \leq  |\scb |^{2} ,
$$
 $$
  \int_{0}^{T} |\scc_{t} |_{\scc}\,dt 
  \leq  |\scc |.
 $$

 In this section we are dealing with the problem
$$
du_{t}=\big[ A_{t}u_{t} +\sca^{\ast}_{t}u_{t} +\sca_{t}u_{t}
 + \scc_{t}u_{t} 
+f^{*}_{t} + f_{t}+g_{t} \big]\,dt
$$
\begin{equation}
                                   \label{4.28.1}
+\big( B^{k}_{t}u_{t}+\scb^{k}_{t}u_{t} +h^{k}_{t}\big) \,dw^{k}_{t},\quad
t\leq T,\quad u_{t}\big|_{t=0}=u_{0}
\end{equation}
solution of which is defined   similarly
to Definition \ref{definition 4.25.1}.

\begin{definition}
                             \label{definition 6.21.1}
An $H$-valued, $\cF_{t}$-adapted,
$H$-continuous in $t$ function $u_{t}$ on $\Omega\times[0,T]$
is called an  
 $H$-solution of \eqref{4.28.1} if
$u_{\cdot}\in V$ for almost all $(\omega,t)$
$$
\int_{0}^{T}  | u_{t}  | _{V}^{2}\,dt<\infty
$$
and for any $ v \in V$, (a.s.) for all $t\in[0,T]$,  
$$
( v ,u_{t})_{H}=( v ,u_{0})_{H}+\sum_{k}\int_{0}^{t}
\big( v ,B^{k}_{s}u_{s}+\scb^{k}_{s}u_{s}+h^{k}_{s}\big)_{H}\,dw^{k}_{s}
$$
$$ 
+\int_{0}^{t}\big[\langle v ,A_{s}u_{s}
 +\sca^{\ast}_su_{s} 
+f^{*}_{s}\rangle
+( v ,\sca_{s}u_{s}
 +\scc_su_{s}  
+f_{s}+g_{s})_{H}\big]\,ds.
$$
\end{definition}

Take any $\varepsilon>0$ and choose a predictable
$\alpha_{t}\geq \varepsilon$ such that
$$
(16/\delta)|\sca_{s}|^{2}_{\sca}\leq \alpha_{s}, 
\quad
(32/\delta)|\sca_s^{\ast}|^{2}_{\sca^{\ast}}
+(2056+32K_0^2/\delta)|\scb_s|^2_{\scb}
+4|\scc_{s}|_{\scc}\leq  \alpha_{s},   
$$
\begin{equation}
                                                                                      \label{6.27.5.24}
\sup_{\Omega}\int_{0}^{T}\alpha_{t}\,dt<\infty.  
\end{equation}
Set
$$
\phi_{t}=\int_{0}^{t} \alpha_{s}\,ds.
$$
 
Recall Definition \ref{definition 4.20.1} of the Banach space $\bV$, and 
define also the Hilbert space $\bH$ of $H$-valued 
predictable functions $h_t$ on $[0,T]\times\Omega$ 
such that 
$$
|h|_{\bH}^2:=E\int_0^T|h_t|_{H}^2\,dt<\infty.
$$

\begin{theorem}
                               \label{theorem 4.20.1}
Under the above assumptions, for any $H$-valued 
$\cF_{0}$-measurable $u_{0}$
such that $E  | u_{0}  | ^{2}_{H}<\infty$, there exists
an $H$-solution of \eqref{4.28.1} such that

(i) $u_{\cdot} \in\bV $,

(ii) we have
\begin{equation}
                                        \label{4.21.1}
 |u_{\cdot}e^{- \phi_{\cdot}} |_{\bV}^{2}
  +|\alpha^{1/2}_{\cdot}u_{\cdot}e^{- \phi_{\cdot}}|^2_{\bH} 
\leq N E |u_{0} |^{2}_{H}+NJ,
\end{equation}
where 
$$
J=E\Big(\int_{0}^{T}e^{-  \phi_{t}} |g_{t} |_{H}\,dt\Big)^{2}
+E\int_{0}^{T}e^{-2 \phi_{t}}\Big( |f^{*}_{t}  |^{2}_{V^{*}}
+\alpha^{-1}_{t}|f_{t}|^{2}_{H}+
   | h^{\cdot}_{t}  | ^{2}_{\ell_{2}(H)}\Big)\,dt,
$$
and the constants $N$ depend only on $\delta,K_{ 0} $.

Moreover, if we have another $H$-solution 
$u'_{t}$ of
 \eqref{4.28.1} such that $u'_{\cdot}\in \bV$, then 
$$
P(\sup_{t\leq T}  | u_{t}-u'_{t}  | _{H}=0)=1.
$$
\end{theorem}

 \begin{proof}
 First we prove \eqref{4.21.1} as an a priori estimate,
 that is assuming that we are given an $H$-solution $u_{t}$
 such that $u_{\cdot}\in \bV$. This a priori estimate,
 in particular, implies uniqueness, which is the last statement
 of the theorem.
To prove the estimate observe that  
\begin{equation}
                                        \label{4.27.5.24}
        E\Big(\int_{0}^{T} |\sca_{s}u_{s} |_{H}\,ds\Big)^{2}
\leq E\Big(\int_{0}^{T}|\sca_{s}|_{\sca}|u_{s}
 |_{V}\,ds\Big)^{2}
\leq |\sca |^{2} E \int_{0}^{T}  |u_{s}
 |^{2}_{V}\,ds<\infty,
\end{equation} 
$$
E \int_{0}^{T} |\sca^{\ast}_{s}u_{s} |^{2}_{V^{\ast}}\,ds 
\leq 
E \int_{0}^{T}|\sca^{\ast}_{s}|^{2}_{\sca^{\ast}}|u_{s}|^{2}_{H}\,ds 
\leq |\sca^{\ast} |^{2} E\sup_{s\leq T} |u_{s} |^{2}_{H} <\infty,
$$
$$
E \int_{0}^{T} |\scb^{\cdot}_{s}u_{s} |^{2}_{\ell_{2}(H)}\,ds 
\leq E \int_{0}^{T}|\scb_{s}|^{2}_{\scb}|u_{s}
 |^{2}_{H}\,ds 
\leq |\scb |^{2} E\sup_{s\leq T} |u_{s}
 |^{2}_{H} <\infty, 
$$
\begin{equation}
                                                                                 \label{5.27.5.24}
E\Big (\int_{0}^{T} |\scc_{s}u_{s} |\,ds\Big)^2 
\leq E \Big(\int_{0}^{T}|\scc_{s}|_{\scc}|u_{s}|_{H}\,ds\Big)^2 
\leq |\scc |^2 E\sup_{s\leq T} |u_{s}|^{2}_{H} <\infty.  
\end{equation}
Therefore, we can apply Theorem \ref{theorem 4.9.3} 
to \eqref{4.28.1}, and using 

 $$
 2\langle u_{s}, A_{s}u_{s} \rangle+\sum_{k}|B^{k}u_{s}|^{2}
 _{H}\leq -\delta |u_{s}|^{2}_{V},
 $$
 $$
 2(u_{s},f_{s}+\sca_su_s)_{H}\leq(\alpha_{s}/2)|u_{s}|_{H}^{2}
 +4\alpha^{-1}_{s}|f_{s}|^{2}_{H}
 +4\alpha^{-1}_{s}|\sca_{s}|^{2}_{\sca}|u_s|^2_{V},
 $$
 $$
 \leq (\alpha_{s}/2)|u_{s}|_{H}^{2}
 +4\alpha^{-1}_{s}|f_{s}|^{2}_{H}+(\delta/4)|u_s|^2_{V},
 $$
 $$
 2\langle u_{s},  f^{*}_{s}+\sca^{\ast}_{s}u_{s}\rangle
 \leq(\delta/8)|u|^{2}_{V}
 +(16/\delta)|f^{*}_{s}|^{2}_{V^{*}}
 +(16/\delta)|\sca_s^{\ast}|^{2}_{\sca^{\ast}}|u_{s}|^{2}_{H},
 $$
$$
\sum_{k}|B^{k}_{s}u_{s}+\scb_tu_s+h^{k}_{s}|^{2}_{H}-\sum_{k}
 |B^{k}_{s}u_{s}|^{2}_{H}
 $$
 $$\leq
  \sum_{k}2|B^{k}_{s}u_{s}|_{H}
 |h^{k}_{s}+\scb^k_tu_s|_{H}+ 2|h ^{\cdot}_{s}|_{\ell_{2}(H)}^{2}
 +2|\scb_s|^2_{\scb}|u_s|^2_H
 $$
 $$
 \leq  (\delta/8)|u_{s}|^{2}_{V}
 +[(16K_0^2/\delta)+4]|h ^{\cdot}_{s}|_{\ell_{2}(H)}^{2}
 +(16K_0^2/\delta)|\scb_s|^2_{\scb}|u_s|^2,
 $$
 $$
 2(u_{s},\scc_su_s)_{H}\leq 2|\scc_s|_{\scc}|u_s|^2_H,
 $$
we get, by taking into account \eqref{6.27.5.24},  that
 $$
   |u_{t}|_{H}^{2}e^{-2  \phi_{t}}
 +(\delta/2)\int_{0}^{t}e^{-2  \phi_{s}}|u_{s}
 |_{V}^{2}\,ds +  \int_{0}^{t}\alpha_{s}
 |u_{s}|_{H}^{2}e^{-2  \phi_{s}}\,ds $$
 $$
 \leq  \int_{0}^{t}
 e^{-2  \phi_{s}}\big((16/\delta)|f^{*}_{s}
 |_{V^{*}}^{2}+4\alpha^{-1}_{s}|f_{s}|^{2}_{H}\big)\,ds+
 2 \int_{0}^{t}e^{-2  \phi_{s}}
 |(u_{s},g_{s})_{H}|\,ds
 $$
\begin{equation}
                                       \label{1.27.5.24}
  +[(16K_0^2/\delta)+4]  \int_{0}^{t}e^{- 2 \phi_{s}}
|h ^{\cdot}_{s}|_{\ell_{2}(H)}^{2}\,ds +m_{t}, 
\end{equation}
 where
 $$
 m_{t}=  2\sum_{k}\int_{0}^{t}e^{-2  \phi_{s}}
 (u_{s},B^{k}_{s}u_{s}+\scb^{k}_su_s+h^{k}_{s})_{H}\,dw^{k}_{s}.
 $$
 
After that it suffices to repeat
 almost literally the proof of Theorem \ref{theorem 4.18.1}
 after \eqref{6,21.4} with only one change
 related to the fact that now in the estimate
 of $E\sup_{t\leq T}|m_{t}|$ the new term  
 $$
 512 E\int_{0}^{T}|\scb_s|^2_{\scb}|u_s|^2\,ds
 $$
 appears. 
 We estimate it by taking the expectations
 in \eqref{1.27.5.24} and using that $512\leq \alpha/4$.
 Then we follow the rest of the proof of Theorem
 \ref{theorem 4.18.1}. This proves \eqref{4.21.1}.

   Now we prove the existence of a solution to 
\eqref{4.28.1} by the method of continuity. 
For $\lambda\in[0,1]$
let ${\rm Eq}_{\lambda}(f,f^{\ast},g,h)$ denote the Cauchy problem 
\eqref{4.28.1} with $\lambda\sca_t$, 
$\lambda\sca^{\ast}_t$, $\lambda\scb_t$ 
and $\lambda\scc_t$ in place of  
$\sca_t$, $\sca^{\ast}_t$, $\scb_t$ 
and $\scc_t$, respectively. 
Let $S$ be the set of parameters $\lambda\in[0,1]$ such that  
${\rm Eq}_{\lambda}(f,f^{\ast},g,h)$ has a solution  
in $\bV$ for any predictable functions $f$, $f^{\ast}$, $g$, and $h$ 
on $[0,T]\times\Omega\times\bR^d$, 
with values in $H$, $V^{\ast}$, $H$ 
and $\ell_2(H)$, respectively, such that 
$$
E\Big(\int_{0}^{T}  |g_{t} |_{H}\,dt\Big)^{2}
+E\int_{0}^{T}\Big( |f^{*}_{t} |^{2}_{V^{*}}
+|f_{t}|^{2}_{H}+
 |h^{\cdot} _{t} |^{2}_{\ell_{2}(H)}\Big)\,dt<\infty. 
$$
Then $0\in S$ by Theorem \ref{theorem 4.18.1}.  
Fix $\lambda_0\in S$ and take a $\lambda\in[0,1]$. Then, 
due to \eqref{4.27.5.24} and \eqref{5.27.5.24}, for any $v\in\bV$ 
the equation 
\begin{align}
du_{t}=&\big[ A_{t}u_{t} +\lambda_0(\sca^{\ast}_{t}u_{t} +\sca_{t}u_{t}
+\scc_{t}u_{t})\big]\,dt 
+\big( B^{k}_{t}u_{t}+\lambda_0\scb^{k}_{t}u_{t}\big) \,dw^{k}_{t}   \nonumber\\
&+\big[(\lambda-\lambda_0)(\sca^{\ast}_{t}v_{t} +\sca_{t}v_{t}
+\scc_{t}v_{t})
+f^{*}_{t} + f_{t}+g_{t} \big]\,dt                                                \nonumber\\
&+\big( (\lambda-\lambda_0)
\scb^{k}_{t}v_{t} +h^{k}_{t}\big) \,dw^{k}_{t},
\quad
t\leq T,\quad u_{t}\big|_{t=0}=u_{0}                                          \nonumber 
\end{align}
has a unique solution $u=:Q_{\lambda}v\in\bV$. 
By estimate \eqref{4.21.1}, 
taking into account \eqref{4.27.5.24} and \eqref{5.27.5.24}, 
for any $v,v'\in\bV$
we have 
$$
|Q_{\lambda}v-Q_{\lambda}v'|^2_{\bV}
\leq N|\lambda-\lambda_0|^2
|v-v'|^2_{\bV}, 
$$
where $N$ is a constant depending only on $K_0$, $\delta$, 
$|\sca|$, $|\sca^{\ast}|$, $|\scb|$ and $|\scc|$. 
Thus  
for $\lambda\in[0,1]$ such that 
$|\lambda-\lambda_0|\leq (2N)^{-1/2}$,  
the operator $Q_{\lambda}$ is a contraction on $\bV$. 
Consequently, 
for these parameters $\lambda$ the operator $Q_{\lambda}$ 
has a fixed point $u^{\lambda}$, i.e.,  ${\rm Eq}_{\lambda}(f,f^{\ast},g,h)$ 
has a solution $u^{\lambda}\in\bV$.  
Hence it follows that $S=[0,1]$. In particular, $1\in S$, i.e., \eqref{4.28.1} 
has a solution $u\in\bV$,  
which finishes the proof of the theorem. 
 \end{proof}
 
\mysection{Stability property}
                                   \label{section 4.22.1}

Assume that 
for any $n=0,1,2,...$, $t\in[0,T],\omega\in\Omega$ 
we are given
linear operators 
$A_{t}^{n}:V\to V^{*}$, 
$B_{t}^{n}:V\to \ell_{2}(H)$, $\sca^{n}_{t}:V\to H$, 
 $\sca^{\ast n}_{t}:H\to V^{\ast}$,  
$\scb^{n}_{t}:H\to \ell_{2}( H)$, 
 $\scc^n:H\to H$  
such that, for each $v\in V$ and $h\in H$,
 $A_{t}^{n}v,B_{t}^{n}v,\sca^{n}_{t}v,
 \sca^{\ast n}_{t}v,  
 \scb^{n}_{t}h$, 
  and $\scc^{n}_{t}h$   
 are predictable.
Also assume that for any $n=0,1,2,...$ on $ \Omega\times[0,T]$
we are given a predictable $V^{*}$-valued function $f^{*n}_{t} $,
an $H$-valued function $f_{t}^{n}$ and $g^{n}_{t}$, and
  an $\ell_{2}(H)$-valued $h_{t}^{n}=(h^{nk}_{t}
   ,k=1,2,...)$. Here comes a more quantitative assumption.

\begin{assumption}
                                  \label{assumption 4.22.1}
(i) For any $n$, $A_{t}^{n},B_{t}^{n}$ satisfy
Assumption \ref{assumption 4.18.1} (with fixed $\delta,K$).   

(ii) We have
$$
\sup_{n}\int_{0}^{T}  |\sca_{t}^{n}  |^{2}_{\sca}\,dt\leq
  | \sca  |^{2},\quad 
  \sup_{n}\int_{0}^{T}  |\sca_{t}^{\ast n}  |^{2}_{\sca^{\ast}}\,dt\leq
  | \sca^{\ast}|^{2},\quad
  \sup_{n}\int_{0}^{T} 
   |\scb_{t}^{n}  | ^{2}_{\scb}\,dt\leq
  | \scb  | ^{2},
$$
$$
\sup_{n}\int_{0}^{T} |\scc_{t}^{n}  |_{\scc}\,dt\leq
  | \scc  |,
$$

\begin{equation*}
                                \label{4.22.4}  
 \sup_{n} \Big[E\Big(\int_{0}^{T}   | g_{t}^{n}  | _{H}\,dt\Big)^{2}
+E\int_{0}^{T}\Big(  | f^{*n}_{t}   | ^{2}_{V^{*}}+
| f^{n} _{t}   | ^{2}_{H}+
 | h^{n \cdot}_{t}   | ^{2}_{\ell_{2}(H)}\Big)\,dt\Big]                             
 <\infty,
\end{equation*}

(iii)         We have
$$
\lim_{n\to\infty}\Big[E\Big(\int_{0}^{T}
   | g_{t}^{n}-g_{t}^{0}  | _{H}\,dt\Big)^{2}
$$
\begin{equation*}
                                \label{4.22.40}
+E\int_{0}^{T}\Big(  | f^{*n}_{t} -f^{*0}_{t}  
 | ^{2}_{V^{*}}+| f^{ n}_{t} -f^{ 0}_{t}  
 | ^{2}_{H}+
  | h^{n \cdot}_{t} -h^{0 \cdot}_{t}   | ^{2}_{\ell_{2}(H)}\Big)\,dt\Big]=0.
\end{equation*}

(iv) For any $v_{\cdot}\in \bV$   we have 
$$
\lim_{n\to\infty}E\Big(\int_{0}^{T}  | (\sca^{n}_{t}
-\sca^{0}_{t})v_{t}  | _{H} 
\,dt\Big)^{2}=0,\,\,
\lim_{n\to\infty}E\int_{0}^{T}  | (\sca^{\ast n}_{t}
-\sca^{\ast 0}_{t})v_{t}  |^2 _{V^{\ast}} 
\,dt=0,
$$
\begin{equation}
                                                            \label{4.23.1}
\lim_{n\to\infty}E \int_{0}^{T}  | (\scb^{n}_{t}
-\scb^{0}_{t})v_{t}  |^{2} _{\ell_{2}(H)} 
\,dt =0,
\quad
\lim_{n\to\infty}E \Big(\int_{0}^{T}  | (\scc^{n}_{t}
-\scc^{0}_{t})v_{t}  | _{H}\,dt\Big)^2  =0
\end{equation}
and, for any $t\leq T$,  $v\in V'$ 
 for a dense subset 
$V'$ in $V$, 
\begin{equation} 
                                                              \label{5.10.6.24}
\lim_{n\to\infty}E\big(  | (A^{n}_{t}-A^{0}_{t})v 
 | ^{2}_{V^{*}}+
  | (B^{n}_{t}-B^{0}_{t})v  | ^{2}_{\ell_{2}(H)}\big)=0.
\end{equation}
\end{assumption}

\begin{theorem}
                                 \label{theorem 4.22.1}
Under the above assumptions, let $u_{0}^{n}\in 
L_{2}(\Omega,\cF_{0},H)$, $n=0,1,2,...$, be such that
$u_{0}^{n}\to u_{0}^{0}$ as $n\to\infty$ in $
L_{2}(\Omega,\cF_{0},H)$.
Denote by $u_{t}^{n}$
the functions from Theorem \ref{theorem 4.20.1}
corresponding to $u_{0}^{n}$, $A^{n}_{t},B^{n}_{t},
\sca^{n}_{t},\scb^{n}_{t},  \scc^{n}_{t}$,  $f^{*n}_{t},f^{n}_{t},g^{n}_{t}$,
and $h^{n}_{t}$.   
Then $  | u^{n}_{\cdot}-u^{0}_{\cdot}  | _{\bV}\to0$
as $n\to\infty$.

\end{theorem}

Proof. Note first that by (i) the operators 
$A^{n}_t$, $B^{n}_t$, $\sca^{n}_{t}$, $\sca^{*n}_{t}$,
$\scb^{n}_{t}$, $\scc^{n}_{t}$
are uniformly bounded in the respective spaces,
 and hence  \eqref{5.10.6.24} holds 
for every $v\in V$. Set $\Delta^{n}_{t}=u^{n}_{t}-
u^{0}_{t}$ and observe that, for any $v\in V$,
(a.s.) for all $t\in[0,T]$
$$
( v ,\Delta^{n}_{t})_{H}=( v ,\Delta^{n}_{0})_{H}
+\sum_{k}\int_{0}^{t}
\big( v ,B^{nk}_{s}\Delta^{n}_{s}+
\scb^{nk}_{s}\Delta^{n}_{s}+\bar h^{nk}_{s}\big)_{H}\,dw^{k}_{s}
$$
$$
+\int_{0}^{t}\big[\langle v ,A^{n}_{s}\Delta^{n}_{s}
+\bar f^{*n}_{s}\rangle
+( v ,\sca^{n}_{s}\Delta^{n}_{s}
 +\sca^{\ast n}_{s}\Delta^{n}_{s}
  +\scc^n \Delta^{n}_{s} 
+\bar f^{n}_{s}+\bar g^{n}_{s})_{H}\big]\,ds,
$$
where 
$$
\bar h^{nk}_{s}=(B^{nk}_{s}-B^{0k}_{s})u^{0}_{s}+
(\scb^{nk}_{s}-\scb^{0k}_{s})u^{0}_{s}+
 h^{nk}_{s}-
 h^{0k}_{s} ,
$$
$$
\bar f^{*n}_{s}=(A^{n}_{s}-A^{0}_{s})u^{0}_{s}
+(\sca^{\ast n}_{s}-\sca^{\ast 0}_{s})u^{0}_{s}
+
 f^{*n}_{s}-
 f^{*0}_{s} ,
$$
$$
\bar f^{n}_{s}=
 f^{ n}_{s}-
 f^{ 0}_{s},\quad \bar g^{n}_{s}=
 (\sca^{n}_{s}-\sca^{0}_{s})u^{0}_{s} 
  +(\scc^{n}_{s}-\scc^{0}_{s})u^{0}_{s} 
 +
 g^{n}_{s}-g^{0}_{s}.
$$

It follows by Theorem \ref{theorem 4.20.1} (with $\varepsilon=1$
in the construction of $\alpha$) that
$$
|\Delta^{n}_{\cdot}|^{2}_{\bV}
\leq NE\int_{0}^{T}\big(|\bar f^{*n}_{s}|^{2}_{V^{*}}
+|\bar f^{n}_{s}|^{2}_{H}+
|\bar h^{n\cdot}_{s}|^{2}_{\ell_{2}(H)}
\big)\,ds+NE\Big(\int_{0}^{T}|\bar g^{n}_{s}|_{H}
\,ds\Big)^{2},
$$
where the constants $N$ depend 
 only on $|\sca|,|\sca^{*}|,|\scb|, 
  |\scc|, 
 \delta, K_{ 0}$. 
 The reader will easily prove that the right-hand side here
 goes to zero as $n\to\infty$, by using our assumptions and
 the dominated convergence theorem. The theorem is proved.

\mysection{First application to the $L_{2}$-theory
of SPDEs}
                         \label{section 4.24.1}
                         
 We introduce some notation used throughout the rest of the paper. 
 We denote  by $\bR^d$ the $d$-dimensional Euclidean 
 space of points $x=(x^1,...,x^d)$. 
The Borel $\sigma$-algebra on $\bR^d$ is denoted by $\cB(\bR^d)$. 
The space of compactly supported smooth functions on $\bR^d$ is denoted by 
$C_0^{\infty}$. 
We use the notation $D_iu=\frac{\partial}{\partial x^i}u$ 
for generalised derivatives 
of locally integrable functions $u$ on $\bR^d$ 
with respect to the $i$-the coordinate $x^i$,    
$Dh=(D_{j}h;j=1,...,d)$, $D^{2}h=(D_{ij}h:=D_{i}D_{j}h
;ij=1,...,d)$.

Generally, if $\sigma(x)=(\sigma^{i...}_{j...}(x))$,
 by $|\sigma(x)|$ we mean the square root of the sum
 of squares of $\sigma^{i...}_{j...}(x)$, for instance,
 for a function $h$ with values
 $(h_{1},h_{2},...)$
 in a finite-dimensional or infinite-dimensional space, in
 $\bR^{n}$ or $\bR^{\infty}$, given on $\bR^{d}$
 we write
 $$
 |h(x)|^{2}=\sum_{i}|h_{i}(x)|^{2},
 $$ 
 $$
 |h|_{L_{p}}^{p}=\int_{\bR^{d}}|h(x)|^{p}\,dx
 =\int_{\bR^{d}}\Big|\sum_{i=1}^{\infty}|h_{i}(x)|^{2}\Big|^{p/2}\,dx.
 $$
 Of course, if the right-hand side is finite, we write
 $h\in L_{p}$. In particular, $Dh\in L_{p}$ if
$$
 |Dh|_{L_{p}}^{p}=\int_{\bR^{d}}|Dh(x)|^{p}\,dx
 =\int_{\bR^{d}}\Big|
 \sum_{i=1,j=1}^{\infty,d}|D_{j}h_{i}(x)|^{2}\Big|^{p/2}
 \,dx<\infty.
 $$
If $p=2$ in the previous sections
 we used a different notation $\|h\|_{\ell_{2}(H)}$
 for $|h|_{L_{2}}$. This was just because there
 $H$ was just an abstract Hilbert space.
 In the space of $\bR^n$- or $\bR^{\infty}$-functions
 the above notation looks more natural.

 For an integers $m\geq0$ the notation 
$W^m_p$ means the Sobolev space of functions $u\in L_p$ 
such that 
$$
|u|^p_{W^m_p}:=\sum_{k=0}^m|D^ku|^p_{L_p}<\infty, 
$$
where $ D^ku $ is   the collection of 
all generalised derivatives of $u$ of order $k$, $D^0u=u$.

In $\bR^{d}$ we consider the equation 
\begin{equation} 
                                          \label{eq1}
du_t=(\cL_tu_t+D_i\frf^i_t+f_t+g_{t})\,dt+
(\cM^k_tu_t+h^k_t)\,dw^k_t,\quad t\leq T
\end{equation}
with initial condition
\begin{equation}                                  \label{eq2}
u_{t}\big|_{t=0}=u_{0}, 
\end{equation}
where 
$$
\cL_tu_t= D_i(a^{ij}_{t}D_{j}u_t
+\beta^i_tu_t)+b^{i}D_{i} u_t +c_tu_t, 
\quad
\cM^k_tu_t=\sigma^{ik}_{t}D_{i}u_{t}+\nu^{k}_{t}u_{t}.  
$$
Here, and later on  we use the summation convention 
with respect to 
repeated integer valued indices.

The coefficients of the operators $\cL$ and $\cM^k$, and 
the free terms $\frf^i$, $f,g$, 
and $h^k$ in the above equation are assumed to be
$\cP\otimes\cB(\bR^d)$-measurable functions on
 $\Omega\times[0,\infty)\times\bR^d$. 
We assume that  the function $a=(a^{ij}_t(x))$ is   
$d\times d$ matrix-valued; 
$\beta=(\beta^i_t(x))$, $b=(b^i_t(x))$, $\sigma^{k}=(\sigma_t^{ik}(x))$, 
$\frf=(\frf^i_t(x))$ are $\bR^d$-valued;  
 $c=c_t(x)$, $f=f_t(x)$, $g=g_t(x)$, $\nu^k=\nu^k_t(x)$
and $h^k= h^k_t(x) $ are real-valued functions.   

\begin{definition}                                                     
                                     \label{definition L2}
Set $W^{0}_{2}=L_{2}$. For $i=0,1$ an $W^{i}_{2}$-valued
   function $u$ on 
$\Omega\times [0,T]$ is called an $W^{i}_{2}$-solution
 of  \eqref{eq1}--\eqref{eq2}
 if $u_{t}$ is a $W^{i}_{2}$-continuous $\cF_t$-adapted process,  
$$
u_t\in W^{i+1}_{2}\quad\text{for $P\otimes dt$-a.e. 
$(\omega,t)\in \Omega\times [0,T]  $,\quad 
$\int_{0}^{T}|u_t|_{W^{i+1}_{2}}^2\,dt<\infty $ (a.s.)},   
$$  
and, for any  $\varphi\in C^{\infty}_0(\bR^d)$, with probability one 
\begin{align}                                                                                              \label{solution}
(u_t,\varphi)=&(\psi,\varphi)+
\int_0^t(\sigma^{ik}_sD_iu_s+\nu^{k}_su_s+h^{k}_s,\varphi)\,dw^k_{s}\\     \nonumber
&+\int_0^t
\{(b^{i}_sD_iu_s+c_su_s+f_s+g_{s},\varphi)-(a^{ij}_sD_ju_s
+\beta^i_su_s+\frf^i_s,D_i\varphi) \}\,ds                                                 
\end{align}
for all $t\in[0, T]$,  where for functions $h,v$ 
on $\bR^d$ the notation $(h,v)$ means the 
integral of $hv$ over $\bR^d$ against the Lebesgue measure.  
\end{definition}

Fix some
$$
 r\in(2,d],\quad (d\geq3),\quad \rho_0\in(0,1],
$$  
and
use the notation $\bB_{\rho}$ for the set of 
balls in $\bR^d$ with radius $\rho$. 
\begin{definition}
Let $\alpha$ be 1 or $1/2$.
A real-, vector-, or tensor-valued function $f$ defined on 
$\Omega\times(0,\infty)\times\bR^d$ is called an 
$\alpha$-admissible 
function if $f=f^M+f^B$ with $\cP\otimes\cB(\bR^d)$-measurable 
functions $f^M$ and $f^B$,  
  there exits a {\em constant\/} $\hat f\geq0$ such that 
\begin{equation*} 
                                                           \label{morrey}
\left(\dashint_{B_{\rho}}|f_{t}^M(x)|^{\alpha r}
\,dx\right)^{1/r}\leq \hat f^{\alpha}\rho^{-1}
\quad\text{for $t>0$, $B_{\rho}\in\bB_{\rho}$, $\rho\leq \rho_0$}, 
\end{equation*}
and there exists a $\cP$-measurable bounded function $\bar f$ 
on $\Omega\times(0,\infty)$ such that 
$$
\sup_{x\in\bR^d}| f_t^B(x)|\leq \bar f_t 
\quad 
\text{for $t\geq0$, $\omega\in\Omega$,\,\,\,and\,\,\,}
 \sup_{\Omega}\int_0^{\infty}\bar f_t^{2\alpha}\,dt<\infty. 
$$
We say that a function is admissible
if it is   1-admissible.
\end{definition} 
\begin{example}
                             \label{example 4.29.1}
 One easily checks that, if $r<d$,
the function $1/|x|$ is admissible, 
$1/|x|^{2}$ is 1/2-admissible.
\end{example}

For $r\in[1,\infty)$ and $\lambda\geq0$ we denote by 
$E_{r,\lambda}$ the {\it Morrey space} of functions $f$ on $\bR^d$, with values 
in a Euclidean space, such that 
$$
|f|_{r,\lambda}:=\sup_{\rho\in(0,\rho_0],B\in\bB_{\rho}}\rho^{\lambda}
\left(\dashint_{B_{\rho}}|f(x)|^r\,dx\right)^{1/r}<\infty. 
$$
Notice that we do not change $|f|_{r,\lambda}$ 
if in its definition we take the supremum 
over balls of rational radius $\rho$ between $0$ 
and $\rho_0$, centred at points with 
rational coordinates. Hence, if $f$ is an 
$\cS\otimes\cB(\bR^d)$-measurable 
real-valued function on $\Theta\times\bR^d$ 
for a measurable space $(\Theta, \cS)$, 
such that $|f(\theta)|_{r,\lambda}<\infty$ for every $\theta\in\Theta$,  then 
$|f(\theta)|_{r,\lambda}$ is an $\cS$-measurable function of $\theta\in\Theta$. 
 
\begin{assumption}                                                      \label{assumption 1} 
There is a constant $\delta>0$ such that 
\begin{equation}
                            \label{4.30.1}
|a|\leq \delta^{-1},
\quad
(2a^{ij}-\sigma^{ik}\sigma^{jk})\lambda^i\lambda^j\geq \delta |\lambda|^2
\quad
\text{for all $\lambda\in\bR^d$},\omega,t.
\end{equation}
\end{assumption}

\begin{assumption}
                                                      \label{assumption 2} 
The functions $\beta$, $b$,  
  and $\nu$ are admissible, $c$ is 1/2-admissibble.
 
\end{assumption}

We thus allow rather singular $\beta,b,\nu$,
see Example \ref{example 4.29.1}.

\begin{assumption} 
                                                     \label{assumption 3}
The initial condition $u_{0}$ 
is in $L_{2}(\Omega,\cF_0, L_{2})$ and

\begin{equation}
E\int_0^{T}
(|\frf_s|^{2}_{L_{2}}
+|f_s|^{2}_{L_{2}}+|h _s|^{2}_{ L_{2} 
})\,ds+E\Big(\int_{0}^{T}  | g_{s}  | _{L_{2}}
\,ds\Big)^{2}<\infty.
\end{equation}
\end{assumption}

Let $\mu=(\mu_t)_{t\in[0,T]}$ be an arbitrary 
nonnegative predictable process such that 
\begin{equation}
                                               \label{6.24.5.24}
\sup_{\Omega}\int_0^T\mu_t\,dt<\infty.
\end{equation}
 
\begin{theorem}
                              \label{theorem 4.25.1}
Under the above assumptions there exists 
$\theta_{0}=\theta_{0}(d,\delta,r)\in(0,1]$ such that, if
\begin{equation}
                                   \label{4.25.3}
\hat b+\hat \beta+\hat c+\hat \nu\leq \theta_{0},
\end{equation}
then \eqref{eq1}--\eqref{eq2} has a unique
$L_{2}$-solution. For this solution we have
$$
E\sup_{t\leq T}  | u_{t}e^{- \phi_{t}}  | _{L_{2}}^{2}
+E\int_{0}^{T}
  |  u_{t}e^{- \phi_{t}}  | ^{2}_{W^{1}_{2}}\,dt
 +E\int_{0}^{T}
 \alpha_t | u_{t}e^{-\phi_{t}}  | ^{2}_{L_{2}}\,dt
 $$
 $$
\leq N  E| u_{0}  | ^{2}_{L_{2}}
+NE\int_0^{T}
|h _t e^{- \phi_{t}}|^{2}_{ L_{2}} \,dt
+
NE\Big(\int_{0}^{T}  
|   g_{t} e^{- \phi_{t}}  | _{L_{2}}
\,dt\Big)^{2}
$$
 \begin{equation}
                                          \label{4.25.5}
+NE\int_0^{T}
(|\frf_t e^{- \phi_{t}}|^{2}_{L_{2}}
+| f_{t}e^{- \phi_{t} }|^{2}_{L_{2}})\,dt ,
\end{equation}
where
\begin{equation*}                                                              \label{phi}
\phi_{t}= \int_{0}^{t}
 \alpha_s \,ds,
\end{equation*}
with 
$\alpha_t=\lambda(\bar b^2_{t}+\bar \beta^2_{t}+\bar \nu^2_{t}
+\bar c_{t}+\rho_{0}^{-2}+\delta)+\mu_t$, 
and  nonnegative (finite) constants $N=N(\delta)$ and 
$\lambda=\lambda(d,\delta,r)$. 
\end{theorem}
 
We derive this theorem from Theorem
\ref{theorem 4.20.1} after some preparations.
First we need Lemma 2.4 of \cite{K2022b} in which
for functions $u,v$ on $\bR^{d}$
we write $u\prec v$ if the integral of $u$
over $\bR^{d}$ is less than or equal to that
of $v$. If the integrals are equal, we write
$u\sim v$.   
 
 \begin{lemma}
                          \label{lemma 12.17.1}
(i) If $f$ is admissible, then for any $t$ and $u
\in C^{\infty}_{0}$ we have
\begin{equation*}
                               \label{12.17.1}
|f_{t}|^{2}|u|^{2}\prec N(d,r)\hat f^{2}|Du|^{2}
+\big(N(d,r)\rho_{0}^{-2}\hat f^{2}+2  \bar f^{2}_{t} \big)|u|^{2}.
\end{equation*}

(ii) If $f$ is 1/2-admissible, then for any $t$ and $u,v
\in C^{\infty}_{0}$ we have
\begin{equation}
                               \label{12.17.1b}
|f_{t}^{M}| u^{2}  \prec N(d,r)\hat f |Du|^{2}
+ N(d,r)\rho_{0}^{-2}\hat f  |u|^{2}.
\end{equation}
\end{lemma} 
 
\begin{corollary}
                                                              \label{corollary bDu}
Inequalities \eqref{12.17.1} and \eqref{12.17.1b} imply 
that for admissible $b=(b^1,...,b^d)$, 1/2-admissible $c$  and functions 
$u,v\in C_0^{\infty}$ we have 
$$
(v, b_t^{Mi}D_i u)\leq N\hat b(|Dv|_{H}
+\rho_0^{-1}|v|_{H})|Du|_{H},   
$$
$$
|(u,c_{t}^{M}v)|\leq\big||c_{t}^{M}|^{1/2}u\big|_{H}\big|
|c_{t}^{M}|^{1/2}v\big|_{H}
$$
\begin{equation*}
\leq N
\hat c\big(|Du|_{H}+ \rho_{0}^{-1}
  |u|_{H}\big)\big(|Dv|_{H}+ \rho_{0}^{-1}
  |v|_{H}\big)
\end{equation*}
with constants $N=N(d,r)$, where $H=L_{2}$.
\end{corollary}
 
{\bf Proof of Theorem \ref{theorem 4.25.1}}.
In Section \ref{section 4.18.1} set $H=L_{2},V=W^{1}_{2}$.
Define $A_{t}:\,\,V\to V^{*}$ by requiring
\begin{equation}
                                                                \label{1.10.6.24}
\langle u,A_{t}v\rangle=
(u,b^{Mi}_tD_iv +c_{t}^{M}v-N_{0}\rho_{0}^{-2}v
- \delta v  )-
(D_iu, a^{ij}_tD_jv 
+\beta^{Mi}_{t}v  ) 
\end{equation}
to hold for any $u,v\in V$,
where $(\cdot,\cdot)$ is the scalar product in $L_{2}$,
and the constant $N_{0}$, depending only
 on $d,r$,   and $\delta$, will be specified later.
To show that such $A_{t}$ is well defined observe that,  
by virtue of Corollary \ref{corollary bDu}
\begin{equation}
                                \label{4.30.3}
|(u,b^{Mi}_tD_iv)|\leq   |Dv |_{H}
|\,|b^{M}_{t}|u|_{H}
\leq N(d,r) |Dv |_{H}\hat b\big(|Du|_{H}+ \rho_{0}^{-1}
  |u|_{H}\big).
\end{equation}
Similarly, $|(D_{i}u,\beta^{Mi}_{t}v)|\leq 
N(d,r )|Du|_{H}\hat \beta\big(|Dv|_{H}+ \rho_{0}^{-1}
  |v|_{H}\big)$, 
and
$$
|(u,c_{t}^{M}v)|\leq\big||c_{t}^{M}|^{1/2}u\big|_{H}\big|
|c_{t}^{M}|^{1/2}v\big|_{H}
$$
\begin{equation}
                                \label{4.30.4}
\leq N(d,r)
\hat c\big(|Du|_{H}+ \rho_{0}^{-1}
  |u|_{H}\big)\big(|Dv|_{H}+ \rho_{0}^{-1}
  |v|_{H}\big),
\end{equation}
$$
(D_iu, a^{ij}_tD_jv )\leq N(d,\delta)
|Dv|_{H}|Du|_{H}.
$$
Also let $B_{t}:V\to\ell_{2}(H)$ be defined by
\begin{equation}
                                                     \label{2.10.6.24}
B^{k}_{t}v=\sigma^{ik}_tD_iv +\nu^{Mk}_tv. 
\end{equation}
Due to Lemma \ref{lemma 12.17.1} we have
$$
|B_{t}v|_{H}\leq N(\delta)|Dv|_{H}+
N(d,r ) \hat\nu \big(|Dv|_{H}+ \rho_{0}^{-1}
  |v|_{H}\big),
$$
where, according to the notation
introduced in the beginning of the section,
$$
|B_{t}v|^{2}_{H}=|B_{t}v|^{2}_{L_{2}}=\sum_{i=k}^{\infty}
\int_{\bR^{d}}|B^{k}_{t}v(x)|^{2}\,dx.
$$
Next, let
\begin{equation}  
                                 \label{3.10.6.24} 
\sca_{t}v=b^{Bi}_tD_iv, 
\quad
 \scc_tv=c_{t}^{B}v
+ N_{0}\rho_{0}^{-2} v + \delta v, 
\quad 
\scb^{k}_{t}v=\nu^{Bk}_sv ,
\end{equation} 
then
$$
|\sca_{t}v|_{H}\leq  \bar b_t |v|_{V} ,
\quad
 |\scc_tv|_{H}\leq (\bar c_t+N_0\rho_0^{-2}+\delta)|v|_H, 
\quad |\scb _{t}v|_{H}\leq \bar \nu_{t}|v|_{H}, 
$$
$$
|\sca_{t}|_{\sca} \leq \bar b_{t},
\quad
|\scc_{t}|_{\scc} \leq\bar c_{t}+N_{0}\rho_{0}^{-2}+\delta,
\quad 
|\scb_{t}|_{\scb}\leq \bar \nu_{t}.
$$
 Define also $\sca^{\ast}_{t}:H\to V^{\ast}$ such that  
\begin{equation}
                                 \label{6,26.1}
\langle u,\sca^{\ast}_{t}v\rangle=(D_iu, \beta^{Bi}_tv)_H
\end{equation}
holds for all $u\in V$ and $v\in H$. Then 
$$
(D_iu, \beta^{Bi}_tv)_H\leq \bar\beta_t |u|_V||v|_H, 
\quad 
|\sca^{\ast}_{t}|_{\sca^{\ast}}\leq  \bar\beta_t. 
$$
 
Finally, for $v\in V$, due to \eqref{4.30.1},
\eqref{4.30.3}, and \eqref{4.30.4} (with $N_{i}=N_{i}(d,
\delta,r)$ and assuming 
apriori that $\hat b, \hat c,
\hat \beta ,\hat\nu  \leq 1$)
$$
 2\langle v, A_{t}v\rangle+|B _{t}v|^{2}_{ 
H }\leq -\delta|Dv|^{2}_{H}-2N_{0}\rho_{0}^{-2}|v|_{H}^{2}
 -2\delta|v|^2_{H} 
$$
$$
+N (\hat b+\hat \beta+\hat c)|Dv|^{2}_{H}
+N \rho_{0}^{-2}|v|_{H}^{2}+2
|\sigma^{i\cdot}_sD_iv|_{H}|\nu^{M }_{s}v|
_{H}
$$
$$
+|\nu^{M }_{s}v|_{H}^{2}
 -2\delta|v|^2_{H} 
\leq -(\delta/2)|Dv|^{2}_{H}-
2N_{0}\rho_{0}^{-2}|v|_{H}^{2}
$$
$$
+N_{1} (\hat b+
\hat \beta+\hat c+\hat \nu )|Dv|^{2}_{H}
+N_{2} \rho_{0}^{-2}|v|_{H}^{2}
 -2\delta|v|^2_{H}. 
$$
For $N_{2} \leq 2N_{0}$ and
$$
N_{1} (\hat b+
\hat \beta+\hat c+\hat \nu )
\leq \delta/4
$$
we get
$$
 2\langle v, A_{t}v\rangle+|B _{t}v|^{2}_{ 
H }\leq -(\delta/4)|Dv|^{2}_{H}
 -2\delta|v|^2_{H}\leq  -(\delta/4)|v|^{2}_{V},
$$
$$
|B_tv|_{ H }^2\leq N(\delta)|Dv|^2_H+N_2\rho^{-2}|v|^2_H.
$$
After these manipulations Theorem \ref{theorem 4.25.1}
follows immediately  from Theorem
\ref{theorem 4.20.1} 
and Remark \ref{remark 1.10.5.2024}. 

 \begin{example}
                          \label{example 5.12.1}
It turns out that generally $\hat b$ should be  small.
To show this   consider the
function
$$
u_{t}(x)=e^{-|x+w_{t}|^{2}/(2t)},\quad t> 0, \quad u_0=0, 
$$
where $w_{t}$ is a $d$-dimensional Wiener process.
This function satisfies
$$
du_{t}=\big(\Delta u_{t}-\frac{d}{|x+w_{t}|^{2}}(x+w_{t})^{i}D_{i}u_{t}\big)
\,dt +D_{k}u_{t}\,dw^{k}_{t}.
$$
 Also,  
$$
\int_{\bR^{d}}|u_{t}|^{2}\,dx=N(d)t^{d/2},\quad 
\int_{\bR^{d}}|Du_{t}|^{2}\,dx=N(d)t^{d/2-1},
$$
so that $u$ satisfies the requirements in Theorem 
\ref{theorem 4.25.1}. In addition, as we mentioned
above, $|b|$ is admissible. However, \eqref{4.25.5} fails,
which shows that the constant factor $d$ in $b$
is not sufficiently small.

\end{example}

\begin{remark} 
                   \label{remark 1.5.1}

 In the literature  a  very popular condition
on $ b$   is that
$b\in  L_{p ,q  }((0,T)\times\bR^{d})$, that is
$b_{t}(x)=b(t,x)$ (nonrandom) and
\begin{equation}
                                                              \label{1.5.01}
\|b\|_{ L_{p ,q }((0,T)\times\bR^{d})}=
\Big(\int_{0}^{T}\Big(\int_{\bR^{d}}|b (t,x)|^{p}
\,dx\Big)^{q /p }\,dt\Big)^{1/q }<\infty
\end{equation}
with $p ,q \in[2,\infty]$ satisfying
\begin{equation}
                                                              \label{9.25.1}
\frac{d}{p }+\frac{2}{q }=1,
\end{equation}
the so-called Ladyzhenskaya-Prodi-Serrin condition,
see, for instance, \cite{BFGM_19},
\cite{RZ_20}, \cite{RZ_21}, and the references therein. 
  It is worth noting that condition \eqref{9.25.1} appears 
in \cite{P_1959}, \cite{S_1962} and \cite{L_1967} as a critical
 case of a condition 
for the uniqueness of the Leray-Hopf generalised solutions 
to the incompressible Navier-Stokes equation and on their regularity.   
For recent developments in this direction we refer to \cite{DD_09} 
and the references therein.

Observe that, if $p >d$, we can take an
arbitrary constant $\hat N$
and introduce
$$
\lambda(t)=\hat N\Big(\int_{\bR^{d}}
|b (t,x)|^{p }\,dx\Big)^{1/(p -d )},
$$
then for  
$$
b^{M} (t, x)=b  (t, x)I_{|b  (t,x)|\geq \lambda(t)}
$$
and $B\in\bB_{\rho}$  we have
$$
\dashint_{B }|b^{M} (t, x)|^{d}\,dx
\leq \lambda^{d-p }(t)
\dashint_{B }|b  (t, x)|^{p }\,dx
\leq N(d)\hat N^{d-p }\rho^{-d}.
$$
 
Here $N(d)\hat N^{d-p }$ can be made
arbitrarily small if we choose $\hat N$
large enough. In addition, for 
$b^{B} =b -b^{M} $ we have
 $|b^{B} |\leq \lambda $ and
$$
\int_{0}^{T}\lambda^{2}(t)\,dt
=\hat N^{2}\int_{0}^{T}\Big(\int_{\bR^{d}}|b (t,x)|^{p }
\,dx\Big)^{q /p }\,dt<\infty.
$$
This shows that the assumption 
that $b$ is admissible is weaker
than  \eqref{1.5.01}, which is supposed to hold as one of alternative
assumptions
in \cite{RZ_20} and \cite{BFGM_19} if $p>d$.

In case $p=\infty$ and $q=2$, our
assumption on $b$
  is the same as in   \cite{BFGM_19},
just take $b^{ M} =0$  (this case is not 
considered in   \cite{RZ_20}, \cite{RZ_21}), but, if $p =d$ 
(and $q =\infty$) our condition is,  basically, weaker than 
in \cite{RZ_20} (this case is excluded in \cite{RZ_21})
and \cite{BFGM_19}
since we can take $b^{M}=bI_{|b|\geq\lambda}$,
where $\lambda$ is a large constant
and
$$
\int_{\bR^{d} }|b^{M} (t, x)|^{d}\,dx
$$
will be uniformly small if $b(t,\cdot)$
is a continuous $L_{d} $-valued function (one of alternative conditions
in \cite{RZ_20} and \cite{BFGM_19})
or the $L_{d} $-norm of
$b^{M}(t,\cdot)$ is uniformly small  as
in  \cite{BFGM_19}. Our condition on $b$
is satisfied with $r=d$, for instance, if
$$
\lim_{\lambda\to\infty}
\sup_{[0,T]}\int_{\bR^{d}}|b(t,x)|^{d}I_{|b(t,x)|\geq
\lambda \xi(t) }\,dx <\hat b ^{d} 
$$
(with $\hat b$ appropriate for \eqref{4.25.3})
for a  function $\xi(t)$
of class $L_{2}([0,T])$. 

  If $p=d$ another alternative condition
on $b$  in \cite{RZ_20} is that
\begin{equation}
                                               \label{1.31.1}
\lambda^{d}|B\cap\{|b(t,\cdot)|>\lambda\}|
\end{equation}
should be sufficiently small uniformly 
for all $B\in\bB_{1}$, $t$, and $\lambda>0$.
It turns out that in this case 
 the assumption 
that $b$ is admissible is satisfied 
(with $b^{ B} =0$), 
any $r\in (1,d)$, and $\rho_{0}=1$. This is shown in the following way, where
$B\in\bB_{\rho},\rho\leq 1$, $\alpha^{d}=M$,
and $M$ is the supremum of expressions in \eqref{1.31.1}:
$$
\rho^{r}\dashint_{B}|b(t,x)|^{r}\,dx
= N\rho^{r-d}\Big(\int_{0}^{\alpha/\rho}
+\int_{\alpha/\rho}^{\infty}\Big)\lambda^{r-1}
|B\cap\{|b(t,\cdot)|>\lambda\}|\,d\lambda
$$
$$
\leq N\rho^{r-d}\int_{0}^{\alpha/\rho}\rho^{d}
\lambda^{r-1}\,d\lambda
+N\rho^{ r-d}M\int_{\alpha/\rho}^{\infty}\lambda^{r-d-1}
\,d\lambda
=NM^{r/d}.
$$
 
\end{remark}

The case when one has $<$ in place of
$=$
in \eqref{9.25.1} is usually called
subcritical, whereas \eqref{9.25.1}
is a critical case. It turns out 
that assumption 
that $b$ is admissible can be satisfied with
 $r<d$   and $b^{M}(t,\cdot)\not\in L_{r+\varepsilon,\loc}$ 
 no matter how small $\varepsilon>0$ is.
In this sense we are dealing with a ``supercritical'' case.  
Also note that $a$ is a unit matrix   in 
\cite{BFGM_19}, \cite{Ki_23},
\cite{RZ_20} and many other papers,
where it is more appropriate to look
for $W^{2}_{ p }$-solutions, rather that $W^{1}_{2}$-solutions.

\mysection{Second application. Rasing the regularity}

We consider again the Cauchy problem \eqref{eq1}-\eqref{eq2} 
and in addition to Assumption \ref{assumption 1} 
we make the following assumptions. 
\begin{assumption}                                        
                                             \label{assumption 4}
The functions $Da$, $D\sigma$, $b$, $c$, $\nu$ 
and $D\nu$ are 1-admissible. 
\end{assumption}
\begin{assumption}                                        
                                              \label{assumption 5}
We have $\beta=\frf=0$, 
the initial condition $u_{0}$ 
in $L_{2}(\Omega,\cF_0, W^1_{2})$, and
\begin{equation}
E\int_0^{T}
(|f_s|^{2}_{L_{2}}
+|h _s|^{2}_{ W^1_{2} })\,ds<\infty,
\quad
E\Big(\int_{0}^{T}|g_{s}|_{W^{1}_{2}}
\,ds\Big)^{2}<\infty.
\end{equation}
\end{assumption}

\begin{theorem}
                                   \label{theorem 2.5.24}
Let Assumptions \ref{assumption 1}, \ref{assumption 4} and 
\ref{assumption 5} hold. Then there exists 
$\theta_{1}=\theta_{1}(d,\delta,r)\in(0,1]$ such that, if
\begin{equation*}
                                                             \label{1.2.5.24}
\hat b+\hat c+\hat\nu+\widehat{Da}+\widehat{D\sigma}
+\widehat{D\nu}\leq \theta_{1}, 
\end{equation*}
then \eqref{eq1}--\eqref{eq2} has a unique
$L_{2}$-solution that is also a unique
$W^{1}_{2}$-solution  $u=(u_t)_{t\in[0,T]}$, and 
$$
E\sup_{t\leq T}  | u_{t}e^{-\phi_{t}}  | _{W^1_{2}}^{2}
+E\int_{0}^{T}
  | u_{t}e^{-\phi_{t}}  | ^{2}_{W^{2}_{2}}\,dt
\leq N  | u_{0}  | ^{2}_{W^1_{2}}
$$
 \begin{equation}
                                          \label{3.2.5.24}
+NE\int_0^{T}
(|h_t e^{-\phi_{t}}|^{2}_{W^1_{2}}
+|f_{t}e^{-\phi_{t}}|^{2}_{L_{2}})\,dt
+NE\Big(\int_{0}^{T}|
e^{-\phi_{t}}g_{t}|_{W^{1}_{2}}
\,dt\Big)^{2},
\end{equation}
where 
$$
\phi_{t}=N(d,\delta,r)\int_{0}^{t}
(\bar b^{ 2}_{s}+\bar c^{ 2}_{s}
+\bar \nu^{ 2}_{s}
+\overline{Da}^{ 2}_s+\overline{D\sigma}^{ 2}_s
+\overline{D\nu}^{ 2}_s
+\rho_{0}^{- 2 } +\delta ) \,ds 
$$
and the
constants $N$
depend only on $d,\delta $.
\end{theorem}

\begin{proof}
We will again use Theorem
\ref{theorem 4.20.1} but in contrast to
Section \ref{section 4.24.1} we define $H,V$
differently. Here  
 we set $H=W^1_{2}$, $V=W^{2}_{2}$ and 
define the linear operators 
$$
\text{$A_{t}:\,\,V\to V^{*}$
\quad and \quad 
$B_{t}:\,\,V\to \ell_2(H)$}
$$
by requiring
$$
\langle u,A_{t}v\rangle=
\big((1-\Delta)u,b^{Mi}_tD_iv 
+c_{t}^{M}v-N_{0}\rho_{0}^{-2}v -\delta v\big)
-(D_iu, a^{ij}_tD_jv )
$$
\begin{equation}
                                              \label{6.16.6.24}
-(D_kD_iu,(D_ka)^{Mij}_tD_jv)-(D_kD_iu,a^{ij}_tD_kD_jv)
\end{equation}
and for integers $k\geq1$ 
$$
(u,B^k_tv)_H=(u,\sigma^{ik}_tD_iv +\nu^{Mk}_tv)
+(D_lu, (D_l\sigma) ^{Mik}_tD_iv + (D_l\nu) ^{Mk}_tv)
$$
\begin{equation}  
                                                           \label{1.4.5.24}                                                                   
+(D_lu,\sigma^{ik}_tD_lD_iv +\nu^{Mk}_t D_lv)
\end{equation}
to hold for any $u,v\in V$,
where $(\cdot,\cdot)$ and $(\cdot,\cdot)_{H}$ 
are the scalar products in $L_{2}$ 
and in $H=W^1_2$, respectively, and $N_{0}=N_0(d,r)$ 
is a constant, specified later. 

Note that by Lemma \ref{lemma 12.17.1}
$$
((1-\Delta)u,b^{Mi}_tD_iv)\leq |u|_{V}|b^{Mi}_tD_iv|_{L_2}
\leq N\hat b|u|_{V}(|D^2v|_{L_2}+\rho_0^{-1}|Dv|_{L_2}), 
$$
\begin{equation*}                                                                      \label{2.3.5.24}
((1-\Delta)u,c_{t}^{M}v )
\leq N\hat c|u|_{V}(|Dv|_{L_2}+\rho_0^{-1}|v|_{L_2}),  
\end{equation*}
$$
(D_kD_iu,(D_ka)^{Mij}_tD_jv)\leq N\widehat{Da}|u|_V
\big(|D^2v|_{L_2}+\rho_0^{-1}|Dv|_{L_2}\big) 
$$
with a constant $N=N(d,r)$. Hence it is easy to see that 
$A_t$ is a bounded linear operator from 
$V$ into $V^{\ast}$ with operator norm bounded by a constant $N$ 
depending only on $d$, $r$, $\delta$, $\hat b$, $\hat c$, $\widehat{Da}$ 
and $\rho_0$. 
Moreover, for $v\in V$ (we accept apriori
that $\hat b+\hat c+\hat\nu+\widehat{Da}+\widehat{D\sigma}
+\widehat{D\nu}\leq1$)
$$
\langle v,A_{t}v\rangle
\leq -(D_iv, a^{ij}_tD_jv )
-(D_kD_iv,a^{ij}_tD_kD_jv)
-(N_0\rho_{0}^{-2}+\delta)|v|_H^2
$$
\begin{equation}
                                                                 \label{2.4.5.24}
+(\delta/8)(|Dv|^2_{L_2}+|D^2v|^2_{L_2})
+N_1(\hat b+\hat c+
\widehat{Da})|v|_V
+N_2\rho_{0}^{-2}|v|^{ 2}_{H}
\end{equation}
with constant $N_1$, $N_2$ depending only of $d,r,\delta$. 
From \eqref{1.4.5.24} we get 
$$
|B_tv|^2_{ H }\leq |\sigma^{i\cdot}D_iv|^2_{ L_2 }
+|\sigma^{i\cdot}DD_iv|^2_{ L_2 }+
\delta|v|^2_H
$$
$$
+N|\nu^{M }_tv|^2_{ L_2 }
+N|(D\sigma)^{Mi\cdot}_tD_iv|^2_{ L_2 }
$$
$$
+N|\nu^{M\cdot}_tDv|^2_{ L_2 }
+N|(D\nu)^{M\cdot}_tv|^2_{ L_2 }
$$
$$
\leq |\sigma^{i\cdot}D_iv|^2_{ L_2 }
+|\sigma^{i\cdot}DD_iv|^2_{ L_2 }+
\delta|v|^2_H
$$
\begin{equation}                                                   
                                                                    \label{3.4.5.24}
+N (\hat\nu+\widehat{D\sigma})|D^2v|_{L_2}
+N (\hat\nu+\widehat{D\nu})|Dv|^2_{L_2}
+N \rho_0^{-2}|v|^2_H, 
\end{equation}
with the  constants $N $  depending only on $d$, $\delta$ and $r$. 
This shows, in particular, that $B_t$ is a bounded linear operator 
mapping $V$ into $\ell_2(H)$ with operator norm bounded 
by a constant depending only on 
$d$, $\delta$, $r$, $\widehat{D\sigma}$, $\hat\nu$, $\widehat{D\nu}$ 
and $\rho_0$. 

Finally, for $v\in V$, due to \eqref{4.30.1}
\eqref{2.4.5.24}, and \eqref{3.4.5.24} (with $N_{i}=N_{i}(d,\delta,r)$)
$$
 2\langle v, A_{t}v\rangle+|B^{\cdot}_{t}v|^{2}_{\ell_{2}(H)}
 \leq -\delta(|Dv|^2_{L_2}+|D^2v|^{2}_{L_2})
 -2N_{0}\rho_{0}^{-2}|v|_{H}^{2}-\delta |v|^2_{H}
$$ 
$$
+N_1(\hat b+\hat c
+\widehat{Da}+\hat\nu
+\widehat{D\nu}
+\widehat{D\sigma})|v|^2_V
+N_2\rho_0^{-2}|v|^2_H
$$
 and
$$
|B_tv|^2_{\ell_2(H)}\leq N(\delta)|v|^2_{W^2_2} 
+N_1 (\hat\nu+\widehat{D\sigma}+\widehat{D\nu})|D^2v|_{L_2}
+N \rho_0^{-2}|v|^2_H, 
$$
 
Hence for $N_{2}\leq 2N_{0}$ and
$$
N_1(\hat b+\hat c+\widehat{Da}+\hat\nu
+\widehat{D\nu}+\widehat{D\sigma})
\leq \delta/2
$$
we get
$$
 2\langle v, A_{t}v\rangle+|B^{\cdot}_{t}v|^{2}_{\ell_{2}(
H)}\leq -(\delta/2)|v|^{2}_{V} 
$$
 and
$$
|B_tv|^2_{\ell_2(H)}\leq N(d,\delta)|v|^2_{W^2_2} 
+N(d,\delta,r) \rho_0^{-2}|v|^2_H.  
$$

Next, we define the linear operators $\sca^{*}_{t}:H\to V^{*}$,
$\scb_{t}: H\to \ell_2(H)$ 
 and $\scc_t: H\to H$  such that 
for all $u \in V$  and $v\in H$ 
$$
\langle u,\sca^{*}_{t}v\rangle
=\big((1-\Delta)u,b^{Bi}_tD_iv +c_{t}^{B}v
 \big)_{L_{2}}
-(D_kD_iu,(D_ka)^{Bij}_tD_jv)_{L_{2}}, 
$$     
$$
(u,\scb^k_tv)_H=(u,\nu^{Bk}_tv)_{L_{2}}
+(D_lu,(D_l\sigma)^{Bik}_tD_iv +(D_l\nu)^{Bk}_tv)_{L_{2}}
$$
\begin{equation*}
                                                  \label{8.16.6.24}
+(D_lu,\nu^{Bk}_t D_lv)_{L_{2}} , 
\end{equation*}
and 
\begin{equation*}                                     \label{9.16.6.24}
(u,\scc_tv)_H=\big((1-\Delta)u,N_{0}\rho_{0}^{-2}v+\delta v\big)_{L_{2}}
\end{equation*} 
hold. Then it is easy to see that with $N=N(d)$ 
$$
 |\sca^{*} _{t}v|_{V^{*}}\leq N(\bar b_t+\bar c_t
+\overline{Da}_t)|v|_{H}, 
$$
$$
|\scb_t v|_{ H }\leq N(\bar\nu_t+\overline{D\sigma}_t+
\overline{D\nu}_t)|v|_H, 
\quad|\scc_tv|_{H}\leq (N_{0}\rho_{0}^{-2}+\delta) |v|_H .
$$

Let us now consider the evolution equation 
\begin{equation}
                                                                   \label{5.4.5.24}
dv_t=(A_tv_t+\sca^{\ast}_tv_t+F_t^{\ast}+G_{t})\,dt 
+(B^k_tv_t+\scb^{k}_t v_t+h^k_t)\,dw^k_t, 
\quad
v_0=u_0,
\end{equation}
where $F^{\ast}_t\in V^{\ast}$  and
$G_{t}\in H$
are defined by requiring that for all $u\in W^2_2$ 
\begin{equation} 
                                                                    \label{10.16.6.24}
\langle u,F^{\ast}_t\rangle=((1-\Delta)u,f_t),
\quad (u,G_{t})_{H}=((1-\Delta)u,g_{t})
\end{equation}
hold. Then clearly, 
$|F_t|_{V^{\ast}}=|f_t|_{L_2}$,
$|G_{t}|_{H}=|g_{t}|_{H}$, and it is not difficult to see 
that the $H$-solution of this initial value problem, in the sense of 
Definition \ref{definition 4.25.1}, is an $L_2$-solution to 
the Cauchy problem \eqref{eq1}-\eqref{eq2} (with 
$\frf=\beta=0$), 
which by Theorem \ref{theorem 4.25.1} is unique. 
Hence we can finish the proof of the theorem 
by applying Theorem \ref{theorem 4.20.1} 
 and Remark \ref{remark 1.10.5.2024} 
to the evolution equation 
\eqref{5.4.5.24}. 
\end{proof}
Instead of Assumptions \ref{assumption 4} 
and \ref{assumption 5} 
we may make the following assumptions. 
\begin{assumption}                                        
                                                     \label{assumption 6}
The functions $Da$, $D\sigma$, $b$, $\nu $ 
and $D\nu $ are 1-admissible. 
Moreover, $c$ and $Dc$ are 1/2-admissible. 
\end{assumption}
\begin{assumption}                                        
                                                      \label{assumption 7}
We have $\beta=\frf=0$, $f=0$,  
the initial condition $u_{0}$ is 
in $L_{2}(\Omega,\cF_0, W^1_{2})$, and

\begin{equation}
E\Big(\int_0^{T}
|g_s|_{W^1_{2}}\,ds\Big)^2
+E\int_0^T|h _s|^{2}_{ W^1_{2} }\,ds<\infty.
\end{equation}
Following the proof of Theorem \ref{theorem 2.5.24} with 
minor modifications we can obtain the following theorem 
from Theorem   \ref{theorem 4.20.1}.
\end{assumption}
\begin{theorem}
                                                        \label{theorem 3.5.24}
Let Assumptions \ref{assumption 1}, \ref{assumption 6}, and 
\ref{assumption 7} hold. Then there exists 
$\theta_{2}=\theta_{2}(d,\delta,r)\in(0,1]$ such that, if
\begin{equation*}
\hat b+\hat c+\hat\nu +\widehat{Dc}+\widehat{Da}+\widehat{D\sigma}
+\widehat{D\nu }\leq \theta_{1}, 
\end{equation*}
then \eqref{eq1}--\eqref{eq2} has a unique
$L_{2}$-solution $u=(u_t)_{t\in[0,T]}$ and 
$$
E\sup_{t\leq T}  | u_{t}e^{- \phi_{t}}  | _{W^1_{2}}^{2}
+E\int_{0}^{T}
  | u_{t}e^{- \phi_{t}}  | ^{2}_{W^{2}_{2}}\,dt
\leq N  | u_{0}  | ^{2}_{W^1_{2}}
$$
 \begin{equation}
                                                                       \label{1.10.9.24}
 +NE\int_0^{T}
(|h_t e^{- \phi_{t}}|^{2}_{W^1_{2}}\,dt
+NE\Big(\int_0^T|g_{t}e^{- \phi_{t}}|_{W^1_{2}})\,dt\Big)^2,
\end{equation}
where 
$$
\phi_{t}= \lambda \int_0^t
(\bar b  ^2 _{s}+\bar \nu_s^2 
+\overline{Da}^2 _s+\overline{D\sigma}^2 _s
+\overline{D\nu}^2 _s
+\bar c_{s}
 +\overline{Dc}_s
+\rho_{0}^{- 2}  +\delta ) \,ds,    
$$
and $\lambda =\lambda(d,\delta,r) $ and 
$N =N(d,\delta) $ are (finite)
constants.  
\end{theorem}

 \mysection{Estimates of solutions in $L_p$-spaces}

Fix a  $p>2$. In this section we suppose that together with Assumptions 
\ref{assumption 1}, \ref{assumption 2} and \ref{assumption 3} 
the following assumption holds. 

\begin{assumption}     
                                                            \label{assumption p}
The initial condition $u_{0}$ 
is in $L_{p}(\Omega,\cF_0, L_{p})$ and

\begin{equation}
E\Big(\int_0^{T}
(|\frf_s|^{2}_{L_{p}}
+|f_s|^{2}_{L_{p}}+|h _s|^{2}_{L_{p}  })\,ds
\Big)^{p/2}+E\Big(\int_{0}^{T}  | g_{s}  | _{L_{p}}
\,ds\Big)^{p}<\infty.
\end{equation}
\end{assumption}

\begin{theorem}
                                                              \label{theorem p}
Let Assumptions 
\ref{assumption 1}, \ref{assumption 2}, \ref{assumption 3} 
and \ref{assumption p} hold. Then there is 
$\kappa=\kappa(d,p,r,\delta)\in(0,1]$ such that if 
$$
\hat b+\hat\beta+\hat c+\hat\nu\leq \kappa,
$$
then \eqref{eq1}--\eqref{eq2} has a unique
$L_{2}$-solution $u=(u_t)_{t\in[0,T]}$ which  is an $L_p$-valued 
weakly continuous process. Moreover, it is an $L_q$-valued 
continuous process for every $q\in[2,p)$, and there is a constant 
$N=N(d,p,\delta)$   such that 
$$
E\sup_{t\leq T}  | u_{t}e^{-\phi_{t}}  | _{L_{p}}^{p}
+E\int_{0}^{T}
 e^{-p\phi_{t}}  |\,|u_{t}|^{p/2-1} Du_{t} | ^{2}_{L_{2}}\,dt
 $$
 \begin{equation}
                                                      \label{4.24.5.24}
 +E\int_{0}^{T}
 \alpha_t e^{-p\phi_{t}} | u_{t}e^{-\phi_{t}}  | ^{p}_{L_{p}}\,dt
\leq 
N I,
\end{equation}
where
$$
I:=  E| u_{0}  | ^{p}_{L_{p}}
+ E\Big(\int_0^{T}
|h _t e^{-\phi_{s}}|^{2}_{L_{p} }\,dt\Big)^{p/2}+
E\Big(\int_{0}^{T}  
|g_{t} e^{-\phi_{t}}| _{L_{p}}
\,dt\Big)^{p}
$$
\begin{equation}
                                                \label{6,24.3}
+ E\Big(\int_0^{T}
(|\frf_t e^{-\phi_{t}}|^{2}_{L_{p}}
+| f_{t}e^{-\phi_{t} }|^{2}_{L_{p}})\,dt\Big)^{p/2},
\end{equation}
\begin{equation} 
                                                              \label{1.22.9.24}                                                           
\phi_{t}=\int_{0}^{t}
\alpha_s\,ds,
\quad
\alpha_t=\lambda\gamma_t+\mu_t,  
\quad
\gamma_t= \bar b^2_{t}+\bar \beta^2_{t}+\bar \nu^2_{t}+\bar c_{t}
+\rho_{0}^{-2}+1 
\end{equation}
with a constant $\lambda=\lambda(d,p,r,\delta)>0$ 
and any nonnegative 
predictable process $\mu$ satisfying \eqref{6.24.5.24}.
\end{theorem}

To raise the regularity of the solution we make the following 
assumptions.

\begin{assumption}                                                    \label{assumption Dp8}                
The functions $Da$, $D\sigma$, $b$, $c$, $\nu$ 
and $D\nu$ are admissible. 
\end{assumption}

\begin{assumption}                                                         \label{assumption Dp9}                
We have $\beta=\frf\equiv 0$ (a.s.) and 
for $\frp=2,p$ the initial condition $u_{0}$ 
is in $L_{\frp}(\Omega,\cF_{0},  W^1 _{\frp})$  
and 
\begin{equation}
E\Big(\int_0^T(|f_s|^{2}_{L_{\frp}}+
|h _s,Dh _{s}|^{2}_{L_{\frp}  }\,
ds\Big)^{\frp/2}
+E\Big(\int_0^T|g_s|_{W^1_{\frp}}\,ds\Big)^{\frp}<\infty. 
\end{equation}
\end{assumption}
We use a
weight function $\exp(-\Psi_t)$ with 
$$
\Psi_t=\int_0^t\Lambda_s\,ds, 
\quad
\Lambda_s
=C(1+\overline{Da}_s^2+\overline{D\sigma}_s^2+\overline b_s^2
+\overline{c}_s^2
+\overline{\nu}_s^2
+\overline{D\nu}_s^2)+\mu_s, 
$$
where $C=C(d, p, r, \delta, \rho_0)$ is a nonnegative constant, 
and $\mu=(\mu_t)_{t\in[0,T]}$ is a nonnegative 
predictable process satisfying  \eqref{6.24.5.24}.

\begin{theorem}                                                         \label{theorem Dp}
Let Assumptions \ref{assumption 1}, \ref{assumption Dp8}
and \ref{assumption Dp9} hold.  
Then there is a constant $\kappa=\kappa(p,\delta,d,r)>0$ 
such that, if 
\begin{equation*}
                                                                              \label{kappa1}
\widehat{Da}+\widehat{D\sigma}+\hat b+\hat{c}+\hat{\nu}
+\widehat{D\nu}\leq \kappa,  
\end{equation*}
then there is a unique 
$L_{2}$-solution that is also a unique
$W^{1}_{2}$-solution  $u$ to \eqref{eq1}--\eqref{eq2}. 
Moreover, $u=(u_t)_{t\in[0,T]}$ is a $W^1_p$-valued 
weakly continuous process, it is continuous 
as a $W^1_q$-valued process 
for every $q\in[2,p)$, and we have 
\begin{equation*}
E\sup_{t\leq T}|e^{-\Psi_t}u_t|^p_{W^1_p}
+E\int_0^{T}
e^{-p\Psi_t}(\big||Du_t|^{p/2-1} D^2u_t \big|^2_{L_2}
+\Lambda_t|u_t|^p_{W^1_p})\,dt
\end{equation*}
$$
\leq NE|u_0|^p_{W^1_p}
+ E\Big(\int_0^T|e^{-\Psi_t}g_t|_{ W^1_{p} }\,dt\Big)^{p}
$$
\begin{equation} 
                                                             \label{estimate Dp}
+NE\Big(\int_0^T(|e^{-\Psi_t}f_t|^{2}_{ L_{p} }
+|e^{-\Psi_t}h_t|_{W^1_p}^2)\,dt\Big)^{p/2} , 
\end{equation}
where $N$ is a constant depending only 
on $d,p,  \delta$. 
\end{theorem}

{\bf Proof of Theorem \ref{theorem p}}.
By Theorem \ref{theorem 2.5.24} there is a unique   
$L_{2}$-solution   $u$ to \eqref{eq1}--\eqref{eq2}. To prove 
\eqref{4.24.5.24} first we make the additional assumptions that 
the coefficients are smooth in $x\in\bR^d$, their derivatives in $x$ 
are bounded functions on $[0,T]\times\Omega\times\bR^d$, 
and the initial value $u_0$ and free terms $f$, $\frf$, $g$ and $h$ 
are also smooth in $x$ and such that 
$$
E|D^nu_0|^{p}_{L_p}+E\int_0^T|D^nf_t|^p_{L_p}\,dt
+E\int_0^T|D^n\frf_t|^p_{L_p}\,dt
$$
\begin{equation}
                             \label{1.9.6.24}
+E\int_0^T|D^ng_t|^p_{L_p}\,dt
+E\int_0^T|D^nh_t |^p_{L_p}\,dt<\infty
\end{equation}
for every integer $n\geq0$. 
Then by Theorem 1.2 in \cite{KR_77} $u$ is a $W^n_p$-valued continuous 
process. 
Thus we can apply Theorem 2.1  in \cite{K_10} on It\^o's formula to 
$|u_t|^p_{L_p}$, to get 
$$
(1/p)d|u|^p_{L_p}
=-(p-1)\big(u_t|^{p-2}(a^{ij}_tD_ju_t+\beta^i_t u_t
+\frf_t^i),D_iu_t\big)\,dt
$$
\begin{equation}                                                \label{1.1.7.24}
+(\theta_t|u_t|^{p-1},b^i_tD_iu_t+c_tu_t+f_t+g_t)\,dt
+\Big(\int_{\bR^d}J_tdx\Big)dt+dm_t, 
\end{equation}
where $\theta_t=\sign u_t$, 
$$
J_t:=\tfrac{p-1}{2}|u_t|^{p-2}\sum_{k}|M^ku_t+h^k_t|^2
\leq \tfrac{p-1}{2}|u_t|^{p-2}\sigma_t^{ik}D_iu_t\sigma_t^{jk}D_ju_t
$$
\begin{equation}                                                                            \label{7.24.5.24}                                                      
+(\delta/(2p))|DU_t|^2+N'|\nu_t|^2|U_t|^2+N|u_t|^{p-2}|h_t|^{2}, 
\end{equation}
$$
m_t:=\int_0^t\int_{\bR^d}\theta_t|u_t|^{p-1}(M^ku_t+h^k_t)\,dx\,dw^k_t, 
$$
where $U_t= |u_t|^{p/2}$
and $N$, $N'$ are constants 
depending only on $d$, $p$ and $\delta$.  
Using Assumption \ref{assumption 1} we have 
$$
I_1:=-(p-1)|u_t|^{p-2}a^{ij}_tD_ju_tD_iu_t
$$
$$
+\tfrac{p-1}{2}|u_t|^{p-2}\sigma^{ik}D_iu_t\sigma_t^{jk}D_ju_t
+(\delta/(2p))|DU_t|^2
$$
\begin{equation}
                                                                   \label{1.25.5.24}
\leq -\tfrac{2(p-1)}{p^2}(2a^{ij}-\sigma^{ik}\sigma^{jk})D_iU_tD_jU_t
+(\delta/2p)|DU_t|^2\leq -\hat\delta |DU_t|^2
\end{equation}
with $\hat\delta=\delta/(2p)$. By Lemma \ref{lemma 12.17.1} 
(using $\hat b,\hat\beta\leq 1$) we get 
$$
I_2:=-(p-1)|u_t|^{p-2}D_iu_t\beta^i_tu_t
+\theta_t|u_t|^{p-1}b^i_tD_iu_t\leq N|U_t|(|\beta_t|+|b_t|)|DU_t|
$$
\begin{equation*}
\prec (\hat\delta/4) |DU_t|^2+N_1(\hat\beta+\hat b)|DU_t|^2
+N(\bar\beta_t^2+\bar b_t^2+\rho^{-2}_0)U_t^2  
\end{equation*}
with $N_1=N_1(d,p,r,\delta)$ and $N=N(d,p,r,\delta)$. Similarly,  
$$
I_3:=\theta_t|u_t|^{p-1}c_tu_t
\leq |c_t| U_t^2
\prec N_2\hat c|DU|^2
+(\bar c_t+N\rho_0^{-2})U^2, 
$$
with $N_2$ and $N$, depending only on $d$, $p$ and $r$,
 and (using $\hat \nu\leq1$) 
\begin{equation}
                                                                               \label{1.7.6.24}
I_4:=N'|\nu_t|^2|U_t|^2\prec N_3\hat \nu|DU|^2
+(2N'\bar \nu_t^2+N\rho_0^{-2})U^2
\end{equation}
with constants 
$N_3=N_3(d,p,r,\delta)$,  
and $N=N(d,p,r, \delta)$, 
where $N'$ 
is the constant in \eqref{7.24.5.24}. We subject $\hat\beta$,  
$\hat b$, $\hat c$, $\hat\nu$ to
\begin{equation}                                                   \label{3.2.7.24}
N_1(\hat\beta+\hat b)\leq \hat\delta/4, 
\quad
N_2\hat c+N_3\hat \nu\leq   \hat\delta/4
\end{equation}
to get 
\begin{equation}                                                    \label{2.25.5.24}
I_1+I_2+I_3+I_4\prec -(\hat\delta/2) |DU_t|^2
+N(\bar\beta_t^2+\bar b_t^2
+\bar\nu_t^2+\bar c_t+\rho_0^{-2})|U_t|^2 
\end{equation}
with $N=N(d,p,r,\delta)$. Note that 
$$
|u_t|^{p-2}|\frf_t^iD_iu_t|
=|u_t|^{(p/2)-1}||u_t|^{(p/2)-1}\frf_t^iD_iu_t|
=(2/p)||u_t|^{(p/2)-1}\frf^iD_iU_t|.
$$
\begin{equation}                                                              \label{1.2.7.2024}
\leq (2/p)|u_t|^{(p/2)-1}|\frf||DU_t|
\end{equation}
Using the above estimates, \eqref{7.24.5.24}, \eqref{1.25.5.24},  
\eqref{2.25.5.24} and \eqref{1.2.7.2024}, 
from \eqref{1.1.7.24}  
we get 
$$
(1/p)d|u|^p_{L_p}
\leq\Big(\int_{\bR^d}\sum_{i=1}^4I_i(t,x)\,dx)\Big)\,dt
$$
$$
+\big((p-1)\big||u_t|^{p-2}\frf_t^iD_iu_t\big|_{L_1}
+\big||u_t|^{p-1}f_t\big|_{L_1}
+N\big||u_t|^{p-2}|h_t|^{2}\big|_{L_1}\big)\,dt
+dm_t 
$$
$$
\leq -(\hat\delta/2) |DU_t|_{ L_2} ^2\,dt+N(\bar\beta_t^2+\bar b_t^2
+\bar\nu_t^2+\bar c_t+\rho_0^{-2})|U_t|_{ L_2} ^2\,dt
$$
$$
+\big((p-1)\big||u_t|^{p-2}\big||\frf_t|Du_t\big|_{L_1}
+N\big||u_t|^{p-2}|h_t|^{2}\big|_{L_1}\big)\,dt
$$
$$
+\big(\big||u_t|^{p-1}f_t\big|_{L_1}+\big||u_t|^{p-1}g_t\big|_{L_1}\big)\,dt+dm_t. 
$$
Now we take $\lambda$ (in the definition of 
$\alpha$) 
greater than $p$ times  the constant $N$ in 
\eqref{2.25.5.24} ($\lambda$ will be further adjusted later).
 This specifies $\phi_{t}$ and
  using It\^o's product rule we have 
$$
|e^{-\phi_t}u_t|^p_{L_p}
+(\delta/2)\int_0^te^{-p\phi_s}|DU_s|^2_{L_2}\,ds
+\int_{0}^{t}\alpha_s|e^{-\phi_s}u_s|^p_{L_p}\,ds
$$
\begin{equation}
\leq |u_0|^p_{L_p}+\sum_{i=1}^5A_i(t)
\end{equation}
with
$$
A_1(t)=N\int_0^t(|e^{-\phi_s}u_s|^{p-2},
|e^{-\phi_s}h_s|^2 )\,ds, 
$$
$$
A_2(t)=2(p-1)\int_0^te^{-p\phi_s}(|u_s|^{(p/2)-1}|\frf_s|,|DU_s|)\,ds
$$
$$
A_3(t)=
p\int_0^t(|e^{-\phi_s}u_s|^{p-1},|e^{-\phi_s}f_s|)\,ds, 
$$
$$
A_4=p\int_0^t(|e^{-\phi_s}u_s|^{p-1},|e^{-\phi_s}g_s|)\,ds,
\quad
A_5=p\int_0^te^{-p\phi_s}dm_s, 
$$
where $N$ is a constant depending on $d$, $p$ and 
$\delta$.
Hence for any stopping time $\tau\leq T$ we obtain 
$$
E\int_0^{\tau}e^{-p\phi_s}|DU_{s}|^2_{L_2}\,ds
+E\int_{0}^{\tau}\alpha_s|e^{-\phi_s}u_s|^p_{L_p}\,ds
$$
\begin{equation}
                                                           \label{1.6.6.24}
\leq NE|u_0|^p_{L_p}
+N\sum_{i=1}^5EA_i(\tau), 
\end{equation}
\begin{equation}
                                                         \label{2.6.6.24}
E\sup_{t\leq\tau}|e^{-\phi_s}u_s|^p_{L_p}
\leq E|u_0|^p_{L_p}+N\sum_{i=1}^5EA_i(\tau)
+pE\sup_{t\leq\tau}\Big|\int_0^te^{-p\phi_s}\,dm_s\Big|   
\end{equation}
with constants $N=N(d,p,\delta)$.   
By Davis's inequality, 
then 
using $|u|^{p-1}=|u|^{p/2}|u|^{p/2-1}$ and 
Cauchy-Schwarz-Bunyakovsky inequality, for the last term 
we get
$$
pE\sup_{t\leq\tau}\Big|\int_0^{\tau}e^{- p\phi_s }dm_s\Big|
\leq 
3pE\Big(\int_0^{\tau}e^{-2p\phi_s}
\sum_k\big||u_s|^{p-1}(M_s^ku_s+g_s^k)\big|^2_{L_1}\,ds\Big)^{1/2}
$$
$$
\leq 6pE\Big(\int_0^{\tau}e^{-2p\phi_s}
|u_s|_{L_p}^{p}|J_s|_{L_1}\,ds\Big)^{1/2}
$$
$$
\leq 6p\Big(E\sup_{t\leq\tau}|e^{-\phi_s}u_s|_{L_p}^{p}\Big)^{1/2}
\Big(E\int_0^{\tau}e^{-p\phi_s}|J_s|_{L_1}\,ds\Big)^{1/2}
$$
\begin{equation}                                                 \label{1.4.7.24}
\leq \frac{1}{4}E\sup_{s\leq\tau}|e^{-\phi_s}u_s|^p_{L_p}
+{36}p^2 E\int_0^{\tau}e^{-p\phi_s}|J_s|_{L_1}\,ds.  
\end{equation}
From \eqref{7.24.5.24} (see \eqref{1.7.6.24} and \eqref{3.2.7.24}) 
we have   
$$
|J_s|_{L_1}\leq N|DU_s|^2_{L_2}+N_1(\bar\nu_s^2+\rho^{-2}_0)|u_s|^p_{L_p}
+N_2\big||u_s|^{p-2}|h_s|^2 \big|_{L_1}
$$
$$
\leq N|DU_s|^2_{L_2}+\alpha |u_s|^p_{L_p}
+N_2\big||u_s|^{p-2}|h_s|^2 \big|_{L_1}
$$
with constants $N=N(d,p,\delta)$, $N_1=N_1(d,p,r,\delta)$,  
and 
$N_2=N_2(d,p,\delta)$, by choosing $\lambda\geq N_1$.
Using this to estimate the last term in \eqref{1.4.7.24} 
and taking into account \eqref{1.6.6.24}, 
from \eqref{2.6.6.24} we obtain
$$
E\sup_{t\leq\tau}|e^{-\phi_{t}}u_t|^p_{L_p}
+
E\int_0^{\tau}e^{-p\phi_s}|DU_s|^2_{L_2}\,ds
+E\int_0^{\tau}\alpha_s|e^{-\phi_s}u_s|^p_{L_p}\,ds
$$
\begin{equation}                                   \label{2.7.6.24}
\leq  
\frac{1}{4}E\sup_{s\leq\tau}|e^{-\phi_s}u_s|^p_{L_p}
+NE|u_0|^p_{L_p}
+N\sum_{i=1}^4EA_i(\tau)
\end{equation}
%
with a constant $N=N(d,p,\delta)$. 
Clearly, 
$$
NEA_1(\tau)\leq NE\sup_{s\leq\tau}|e^{-\phi_s}u_s|_{L_p}^{p-2}
\int_0^{\tau}|e^{-\phi_s}h_s|^2_{L_p}\,ds
$$
\begin{equation} 
                                                                           \label{3.7.6.24}
\leq(1/8)E\sup_{s\leq \tau}|e^{-\phi_s}u_s|_{L_p}^{p}
+N_1\Big(\int_0^{\tau}|e^{-\phi_s}h_s|^2_{L_p}\,ds\Big)^{p/2},
\end{equation}
$$
NEA_2(\tau)\leq (1/4)E\int_0^{\tau}e^{-p\phi_s}|DU_s|^2_{L_2}\,ds
+N'E\int_0^{\tau}e^{-p\phi_s}\big||u_s|^{(p/2)-1}|\frf_s|\big|^2_{L_2}\,ds
$$
$$
\leq 
(1/4)E\int_0^{\tau}e^{-p\phi_s}|DU_s|^2_{L_2}\,ds
+(1/8)E\sup_{s\leq \tau}|e^{-\phi_s}u_s|_{L_p}^{p}
$$
\begin{equation} 
                                                                              \label{4.7.6.24}
+N_2E\Big(\int_0^{\tau}|e^{-\phi_s}\frf|^2_{L_p}\,ds\Big)^{p/2},
\end{equation}
\begin{equation}
                                                                               \label{5.7.6.24}
NA_3(\tau)\leq E\int_0^{\tau}|e^{-\phi_s}u_s|_{L_p}^{p}\,ds
+N_3E\int_0^{\tau}|e^{-\phi_s}f_s|^p_{L_p}\,ds,
\end{equation}
\begin{equation}                                     
                                                                                  \label{6.7.6.24}
NA_4(\tau)
\leq (1/8)E\sup_{s\leq \tau}|e^{-\phi_s}u_s|_{L_p}^{p}
+N_4E\Big(\int_0^{\tau}|e^{-\phi_s}g_s|_{L_p}\,ds\Big)^{p},  
\end{equation}
where $N$ is the constant from \eqref{2.7.6.24}, 
$N'$, $N_1$, $N_2$, $N_3$ and $N_4$ 
are constants, depending only on 
$d$, $p$ and $\delta$. 
Using now \eqref{2.7.6.24}, and 
\eqref{3.7.6.24} through \eqref{6.7.6.24} with 
$$
\tau_n=\inf\{t\in[0,T]: |u_t|_{W^1_p}\geq n\}\wedge T,
$$
in place of $\tau$, we get 
\eqref{4.24.5.24} with $\tau_n$ instead of 
$T$, which implies \eqref{4.24.5.24} by Fatou's 
lemma, since $\tau_n\to\infty$ as $n\to\infty$. 

To dispense with the additional assumptions on 
the coefficients, initial and free data of the problem 
\eqref{eq1}--\eqref{eq2}, we are going to approximate 
this problem. To this end let 
$\kappa$ be a nonnegative compactly supported 
smooth function with unit integral over $\bR^d$. 
For locally integrable functions $v$ 
on $\bR^d$  we use the notation 
$v^{(\varepsilon)}$ for the mollification $v\ast \kappa_{\varepsilon}$,  
where 
$\kappa_{\varepsilon}(\cdot)
=\varepsilon^{-d}\kappa(\cdot/\varepsilon)$ 
for $\varepsilon\in(0,1)$.  If $h$ is an $\alpha$-admissible function 
with a given decomposition 
$h=h^{B}+h^{M}$ and associated $\hat h,\bar h$,
 then for $\varepsilon>0$  
the notation $h^{\varepsilon}$ means 
\begin{equation}                                                                   \label{eps1}
h^{\varepsilon}:=h^{B\varepsilon}+h^{M(\varepsilon)},  
\quad 
\text{
where
\quad 
$h^{B\varepsilon}:={\bf1}_{\bar h\leq1/\varepsilon}h^{B(\varepsilon)}$}, 
\end{equation}
and the mollification is understood only in the variable $x\in\bR^d$. 
Note that 
\begin{equation}                                                              
                                                                  \label{1.1.4.24}
 \hat{h^{\varepsilon}}
\leq \hat h , 
\quad
|h^{B\varepsilon}|\leq \bar h 
\quad
\text{for all $(\omega,t,x)\in\Omega\times\bR_+\times\bR^d$}. 
\end{equation}
We approximate equation \eqref{eq1} with 
\begin{equation}           
                            \label{eqe}                                                
dv_t=(L^{\varepsilon}_tv_t
+D_i\frf^{\varepsilon i}_t+f^{\varepsilon}_t
+g^{\varepsilon}_t)\,dt
+(M^{\varepsilon k}_tv_t+h^{\varepsilon k}_t)\,dw^k_t,
\quad v_0=u_0^{(\varepsilon)},
\end{equation}
where 
$$
h^{\varepsilon}_t:={\bf1}_{|h_t|_{L_p}\leq1/\varepsilon} 
h^{(\varepsilon)}_t, 
$$
\begin{equation}                                                                   \label{eps2}
\frf^{\varepsilon}_t:={\bf1}_{|\frf_t|_{L_p}
\leq1/\varepsilon}\frf^{(\varepsilon)}_t, 
\quad
f^{\varepsilon}_t:={\bf1}_{|f_t|_{L_p}\leq1/\varepsilon}
f^{(\varepsilon)}_t,
\quad
g^{\varepsilon}_t:={\bf1}_{|g_t|_{L_p}\leq1/\varepsilon}g^{(\varepsilon)}_t, 
\end{equation}
$$
L^{\varepsilon}_tv= D_i(a^{(\varepsilon)ij}_{t}D_{j}v
+\beta^{\varepsilon i}_t v)+b^{\varepsilon i}D_{i} u_t 
+c^{\varepsilon}_tv, 
\quad
M^{\varepsilon k}_tv=\sigma^{(\varepsilon)ik}_{t}D_{i}v
+\nu^{\varepsilon   k}_{t}v.     
$$
It is easy to see that 
the coefficients of this equation together with their 
partial derivatives in $x\in\bR^d$ up to any 
order are bounded functions on $\Omega\times\bR_+\times\bR^d$. 
Moreover, $a^{(\varepsilon)}$ and 
$\sigma^{(\varepsilon)}=(\sigma^{1(\varepsilon)},...,\sigma^{d_1(\varepsilon)})$ 
satisfy Assumption \ref{assumption 1} with the constant $\delta$, and we have 
\eqref{1.9.6.24} with $u_0^{(\varepsilon)}$, $\frf^{\varepsilon}$, 
$f^{\varepsilon}$, $g^{\varepsilon}$ and $h^{\varepsilon}$, in place of  
$u_0$, $\frf$, $f$, $g$ and $h$, respectively. Consequently, by virtue of what we have 
proved above, \eqref{eqe} 
admits a unique $L_2$-solution $u^{\varepsilon}=(u^{\varepsilon}_t)_{t\in[0,T]}$ 
for each $\varepsilon>0$, it is a weakly continuous $L_p$-valued process. 
Moreover,  
$$
E\sup_{t\leq T}  | u^{\varepsilon}_{t}e^{-\phi_{t}}  | _{L_{p}}^{p}
+E\int_{0}^{T}
 e^{-p\phi_{t}}  |\,|u^{\varepsilon}_{t}|^{p/2-1} Du^{\varepsilon}_{t} | ^{2}_{L_{2}}\,dt
 $$
\begin{equation}
                                                  \label{2.9.6.24}
+E\int_{0}^{T}
 \alpha_t e^{-p\phi_{t}}| u^{\varepsilon}_{t}e^{-\phi_{t}}  
 | ^{p}_{L_{p}}\,dt\leq NI,
\end{equation}  
with a constant  $N=N(d,p,\delta)$ and $I$ from \eqref{6,24.3}.

Naturally, we are going to apply Theorem \ref{theorem 4.22.1}
to prove that, for any sequence $\varepsilon_{n}\downarrow 0$,
\begin{equation}
                                              \label{7.9.6.24}
\lim_{n\to\infty}E\sup_{t\in[0,T]}|u^{\varepsilon_n}_t-u_t|^2_{L_2}
+\lim_{n\to\infty}E \int_0^T|Du^{\varepsilon_n}_t-Du_t|^2_{L_2}\,dt=0.                                                                  
\end{equation}
To this end we  set $V=W^1_2$, $H=L_2$ 
and cast \eqref{eq1}--\eqref{eq2} into the evolution equation \eqref{4.28.1}
(see the proof of Theorem \ref{theorem 4.25.1}), and for each integer $n\geq1$ 
we cast \eqref{eqe}, with $\varepsilon_n$ in place of $\varepsilon$,   
into the evolution equation
$$
du^{\varepsilon_n}_{t}=\big[ A^{n}_{t}u^{\varepsilon_n}_{t} 
+\sca^{\ast n}_{t}u^{\varepsilon_n}_{t} 
+\sca^{n}_{t}u^{\varepsilon_n}_{t}
+\scc^{n}_{t}u^{\varepsilon_n}_{t}
+f^{*n}_{t} + f^{n}_{t}
+g^{n}_{t} \big]\,dt
$$
\begin{equation}
                                   \label{4.28.1e}
+\big( B^{nk}_{t}u^{\varepsilon_n}_{t}+\scb^{nk}_{t}u^{\varepsilon_n}_{t}
 +h^{nk}_{t}\big) \,dw^{k}_{t},
 \quad t\leq T,
 \quad 
 u^{\varepsilon_n}_{t}\big|_{t=0}=u^{(\varepsilon_n)}_{0}, 
\end{equation}
(see the proof of Theorem \ref{theorem 4.25.1}), 
where $A^n_t$, $B^n_t$, $\sca_t^{\ast n}$, $\sca_t^{  n}$, 
$\scb_t^{ n}$, 
$\scc_t^{ n}$ are defined as $A_t$, $B_t$, $\sca^{*}_t$, $\sca_t$, $\scb_t$, 
$\scc_t$, but 
with $a^{(\varepsilon_n)}$, $b^{\varepsilon_{n}}$....,$\nu^{\varepsilon_{n}}$
 in place of 
$a$, $b$,....,$\nu$ in their definition in 
\eqref{1.10.6.24}, \eqref{2.10.6.24}, \eqref{3.10.6.24}, and 
\eqref{6,26.1}.
The free terms $f^{\ast n}, f^{n},g^{n}$ 
are defined according to \eqref{eps2},  to ease the
 notation we replace $\varepsilon_{n}$ with $n$. 
To get \eqref{7.9.6.24} we are going to verify the conditions 
of Theorem \ref{theorem 4.22.1}.   

Using well-known properties of mollifications 
and taking into account \eqref{1.1.4.24} we can easily check that 
parts (i), (ii) and (iii) of Assumption \ref{assumption 4.22.1} hold. 
To verify part (iv) of this assumption,  
let $v\in C_0^{\infty}$ be vanishing outside of a ball $B_R$ 
of radius $R$. Then for the unit ball $B_V$ in $V=W^1_2$
$$
\sup_{\phi\in B_V}
\big|\big((a^{ij}_{t}-a^{(\varepsilon)ij}_{t})D_jv,D_i\varphi\big)_{H}\big| 
\leq  \sup_{\bR^d}|Dv||{\bf1}_{B_R}(a-a^{(\varepsilon)})|_{H}, 
$$
$$
\sup_{\phi\in B_V }
|\big((b^{Mi}_t-b_t^{M(\varepsilon)i})D_iv,\varphi\big)_{H}|
\leq \sup_{\bR^d}|Dv|
\big|{\bf1}_{B_R}(b^{M }_t-b_t^{M(\varepsilon) })\big|_{H}, 
$$
\begin{equation}                             \label{6.10.6.24}
\sup_{\phi\in B_V }
|\big((\beta^{Mi}_t-\beta_t^{M(\varepsilon)i})v,D_i\varphi\big)_{H}|
\leq \sup_{\bR^d}|v|
\big|{\bf1}_{B_R}(\beta^{M}_t-\beta_t^{M(\varepsilon)})\big|_{H}.  
\end{equation}
where, by well-known properties 
of mollification, the right-hand side of each inequality converges to zero 
as $\varepsilon\to0$. 
Moreover, using Lemma \ref{lemma 12.17.1} and taking into account 
\eqref{1.1.4.24} we have 
$$
\sup_{\varphi\in B_V }
|\big((c^{M}_t-c_t^{M(\varepsilon)})v,\varphi\big)_{H}|
\leq \sup_{\varphi\in B_V }
|\big(|c^{M}_t-c_t^{M(\varepsilon)}|^{1/2}v,
|c^{M}_t-c_t^{M(\varepsilon)}|^{1/2}\varphi\big)|
$$
$$
\leq \sup_{\bR^d}|v|
\big|{\bf1}_{B_R}|c^{M}_t-c_t^{M(\varepsilon)}|^{1/2}\big|_{L_2}
\sup_{\varphi\in B_V }
\big|{\bf1}_{B_R}|c^{M}_t-c_t^{M(\varepsilon)}|^{1/2}\varphi\big|_{L_2}
$$
\begin{equation}
                                   \label{6,25.1}
\leq N\hat c^{1/2}\sup_{\bR^d}|v|
\big|{\bf1}_{B_R}(c^{M}_t-c_t^{M(\varepsilon)})\big|^{1/2}_{L_1}
\end{equation}
with a constant $N=N(d,r,\rho_0)$. Note that by well-known 
properties of mollification the right-hand side of the 
last inequality converges to zero for $\varepsilon\to0$, 
which together with the convergence to zero
 of the left-hand side of each inequality in \eqref{6.10.6.24} implies 
 $$
 \lim_{n\to\infty}E| (A^{n}_{t}-A_{t})v| ^{2}_{V^{*}}=0. 
 $$
In the same way we have 
$$
\lim_{n\to\infty} | (B^{n}_{t}-B_{t})v  | ^{2}_{\ell_{2}(H)}=0, 
$$
i.e., \eqref{5.10.6.24} holds. 
 To verify \eqref{4.23.1} let $v\in\bV$ (see Definition \ref{definition 4.20.1}) 
  and note that 
\begin{equation} 
                                         \label{7.10.6.24}
E\int_{0}^{T}  | (\sca^{\ast n}_{t}-\sca^{\ast}_{t})v_{t}  |^2 _{V^{\ast}} 
\,dt
\leq 
 E\int_0^T|(\beta^{B}_t-\beta^{\varepsilon_n B}_t)v_t|^2_{H}\,dt, 
\end{equation} 
 where $\lim_{n\to\infty}|\beta^{B}_t-\beta^{\varepsilon_n B}_t|=0$ 
 for  almost every $(\omega,t,x)$, and 
$$
|(\beta^{B}_t-\beta^{\varepsilon_n B}_t)v_t|^2_{H}
\leq 4\sup_{t\in[0,T]}|v_t|^2_{H}\bar\beta^2_t, 
$$
$$
E\int_0^T\sup_{t\in[0,T]}|v_t|^2_{H}\bar\beta^2_t\,dt
\leq E\sup_{t\in[0,T}|v_t|^2_{H}\,
\sup_{\Omega}\int_0^T\bar\beta_t^2\,dt\,<\infty. 
$$
Thus letting $n\to\infty$ in \eqref{7.10.6.24}, 
by dominated convergence 
we get 
$$
\lim_{n\to\infty}
E\int_{0}^{T}  | (\sca^{\ast n}_{t}-\sca^{\ast}_{t})v_{t}  |^2 _{V^{\ast}}\,dt=0.  
$$
Similarly, 
 $$
\lim_{n\to\infty}E\Big(\int_{0}^{T}  | (\sca^{n}_{t}-\sca_{t})v_{t}  | _{H} 
\,dt\Big)^{2}
 \leq T\lim_{n\to\infty}
 E \int_0^T|(b^{B i}_t-b^{\varepsilon_n B i}_t)D_iv_t
 |^{2}_{H}\,dt=0,
 $$
\begin{equation*} 
                                                                                       \label{9.10.6.24}
\lim_{n\to\infty}E\int_{0}^{T}  | (\scb^{n}_{t}-\scb_{t})v_{t}  |^2 _{H}\,dt
=\lim_{n\to\infty}E\int_{0}^{T} 
 \big||\nu^{B}_{t}-\nu^{\varepsilon_n B}_t|_{\ell_2}v_{t} 
  \big|^2 _{H}\,dt=0,
\end{equation*} 
$$
\lim_{n\to\infty}
E\Big(\int_{0}^{T}  | (\scc^{n}_{t}-\scc_{t})v_{t}  | _{H}\,dt\Big)^2
=\lim_{n\to\infty}
E\Big(\int_{0}^{T}  | (c^{B}_{t}-c^{B\varepsilon_n}_{t})
v_{t}  | _{H}\,dt\Big)^2=0,
$$
which finishes the verification of condition \eqref{4.23.1}.
 This also completes 
the verification of the conditions of Theorem \ref{theorem 4.22.1}, and 
by that, the proof of \eqref{7.9.6.24}. 
In particular, for a subsequence of $\varepsilon_n$, for simplicity denoted also 
by $\varepsilon_n$, we have
$$ 
\lim_{n\to\infty}u^{\varepsilon_n}=u, 
\quad
\lim_{n\to\infty}Du^{\varepsilon_n}=Du 
\quad \text{$\Omega\times[0,T]\times\bR^{d}$-(a.e.)}.
$$
By lemma Fatou and \eqref{2.9.6.24} this allows us to conclude
that
$$
E\int_{0}^{T}
 e^{-p\phi_{t}}  |\,|u _{t}|^{p/2-1} Du _{t} | ^{2}_{L_{2}}\,dt
 +E\int_{0}^{T}
 \alpha_t e^{-p\phi_{t}}| u _{t}e^{-\phi_{t}}  | ^{p}_{L_{p}}\,dt
 \leq NI.
 $$
 To get the remaining part of estimate \eqref{4.24.5.24}
 observe that for any constant $M>0$
 $$
 (|u_{t}|\wedge M)^{p}\leq 2^{p}(|u_{t}-u^{\varepsilon_{n}}_{t}|
 \wedge M)^{p}+2^{p}|u^{\varepsilon_{n}}_{t}|^{p}
 $$
 $$
 \leq 2^{p}M^{p-2}|u_{t}-u^{\varepsilon_{n}}_{t}|^{2}
 +2^{p}|u^{\varepsilon_{n}}_{t}|^{p},
 $$
 implying that
 $$
 E\sup_{t\leq T} |u_{t}| ^{p}=\lim_{M\to\infty}
 E\sup_{t\leq T}(|u_{t}|\wedge M)^{p}\leq NI.
 $$
 This proves \eqref{4.24.5.24}.
 
 Finally, the fact that $u_{t}$ is  
weakly continuous as an  $L_p$-valued process and is
continuous as an  $L_q$-valued  process for every $q\in[2,p)$
immediately follows from its continuity as an $L_{2}$-valued
process and the finiteness of $E\sup_{t\leq T} |u_{t}| ^{p}$.
The theorem is proved.

\begin{remark}  
From its proof it can be seen that  
Theorem \ref{theorem p} holds also with the estimate obtained by 
replacing the right-hand side of \eqref{4.24.5.24} with 
$$
N|u_0|^p_{L_p}+NE\int_0^T|e^{- \phi_t}f_t|^{p}_{L_p}\,dt
+NE\Big(\int_0^T|e^{- \phi_t}g_t|^{p/(1+\kappa)}_{L_p}\,dt
\Big)^{1+\kappa}
$$
$$
+NE\Big(\int_0^T|e^{- \phi_t}\frf_t|^{2p/(2+\kappa')}_{L_p}\,dt\Big)^{(2+\kappa')/2}
+NE\Big(\int_0^T|e^{- \phi_t}h_t|^{2p/(2+\kappa'')}_{L_p}\,dt\Big)^{(2+\kappa'')/2}
$$
for any $\kappa\in[0,p-1]$ and $\kappa', \kappa''\in[0,p-2]$, 
where $N$ depends only on $d,p,\delta$.

To see this we need only replace the estimates 
\eqref{3.7.6.24}, \eqref{4.7.6.24} and 
\eqref{6.7.6.24} with the following estimates, 
 valid  for any $\kappa,\kappa'\in[0,p-2]$
 and $\kappa''\in[0,p-1]$,
\begin{align}
NEA_1(\tau)\leq &NE\sup_{s\leq \tau}|e^{-\phi_s}u_s|_{L_p}^{\kappa}
\int_0^{\tau}|e^{-\phi_s}u_s|_{L_p}^{p-2-\kappa}|e^{-\phi_s}h_s|^2_{L_p}\,ds \nonumber\\
\leq &(1/8)E\sup_{s\leq \tau}|e^{-\phi_s}u_s|_{L_p}^{p}
+NE\int_0^{\tau}|e^{-\phi_s}u_s|_{L_p}^{p}\,ds                                              \nonumber\\
&+NE\Big(\int_0^{\tau}
|e^{-\phi_s}h_s |^{2p/(\kappa+2)}_{L_p}\,ds\Big)^{(\kappa+2)/2},     \nonumber\\
NEA_2(\tau)\leq& NE\int_0^{\tau}e^{-p\phi_s}
\big|\frf_s |u_s|^{(p/2)-1}\big|_{L_2}|DU_s|_{L_2}\,ds                                      \nonumber\\                                  
\leq&(1/4)E\int_0^{\tau}e^{-p\phi_s}|DU_s|^2_{L_2}\,ds
+NE\int_0^{\tau}e^{-p\phi_s}|u_s|_{L_p}^{p-2}|\frf_s |_{L_p}^2\,ds            \nonumber\\
\leq &(1/4)E\int_0^{\tau}e^{-p\phi_s}|DU_s|^2_{L_2}\,ds
+(1/8)E\sup_{s\leq \tau}|e^{-\phi_s}u_s|_{L_p}^{p}                                          \nonumber\\
&+NE\Big(\int_0^{\tau}
|e^{-\phi_s}\frf_s|^{2p/(\kappa'+2)}_{L_p}
\,ds\Big)^{(\kappa'+2)/2},                    \nonumber
\end{align}
and 
\begin{align}
NA_4(\tau)\leq &E\sup_{s\leq\tau}|e^{-\phi_s}u_s|_{L_p}^{\kappa''}
\int_0^{\tau}|e^{-\phi_s}u_s|^{p-1-\kappa''}_{L_p}
| e^{-\phi_s} g_s|_{L_p}\,ds                      \nonumber\\
\leq&(1/8)E \sup_{s\leq\tau}|e^{-\phi_s}u_s|_{L_p}
+NE\int_0^{\tau}|e^{-\phi_s}u_s|_{L_p}^p\,ds                                                    \nonumber\\
&+NE\Big(\int_0^{\tau}| e^{-\phi_s} g_s|_{L_p}^{p/(1+\kappa'')}
\,ds\Big)^{1+\kappa''}, 
\end{align}
where $N$ denotes constants depending only on $d$, $p$ and $\delta$.

This has some advantages
in comparison with using H\"older's inequality, which yields, for instance,
$$
E\Big(\int_0^T|e^{- 2\phi_t}g_t| _{L_p}\,dt
\Big)^{p}\leq
E\Big(\int_0^T|e^{- \phi_t}g_t|^{p/(1+\kappa)}_{L_p}\,dt
\Big)^{1+\kappa}
$$
$$
\times \Big(\int_0^Te^{-\phi_{t}p/(p-\kappa-1)}
\,dt
\Big)^{p-\kappa-1}\leq N
E\Big(\int_0^T|e^{- \phi_t}g_t|^{p/(1+\kappa)}_{L_p}\,dt
\Big)^{1+\kappa}.
$$

\end{remark}

{\bf Proof of Theorem \ref{theorem Dp}}.
By Theorem \ref{theorem p} there is a unique $L_2$-solution $u$, and 
estimate \eqref{4.24.5.24} holds. 
To prove estimate \eqref{estimate Dp}, as in the proof of Theorem \ref{theorem p}, 
first we make the additional assumptions that 
the coefficients are smooth in $x\in\bR^d$, their derivatives in $x$ 
are bounded functions on $[0,T]\times\Omega\times\bR^d$, 
and the initial value $u_0$ and free terms $f$, $\frf$, $g$ and $h$ 
are also smooth in $x$ and satisfy \eqref{1.9.6.24}, which imply
(see Theorem 1.2 in \cite{KR_77}) that $u$ is a $W^n_p$-valued continuous 
process for every integer $n\geq1$. To get  \eqref{estimate Dp}                                                              
we follow the proof of a similar estimate in \cite{K2022b}, 
with appropriate changes. To this end first we recall some notations and statements 
from \cite{K2022b}. 

Let $\frh(\eta)=(1+|\eta|^{\kappa})^{-1}$ for $\eta\in\bR^d$ and with a fixed 
$\kappa\geq d+p+1$. For functions $u(x,\eta)$ and $v(x,\eta)$ on $\bR^{2d}$ 
we write $u\prec_{\kappa} v$ if 
$$
\int_{\bR^{2d}}\frh(\eta)u(x,\eta)\,dxd\eta
\leq \int_{\bR^{2d}}\frh(\eta)v(x,\eta)\,dxd\eta.  
$$
Similarly, for functions 
$u_t(x,\eta)$ and $v_t(x,\eta)$ on $[0,T]\times\bR^{2d}$ we write 
$du_t\prec_{\kappa} dv_t$ if $u_t-u_s\prec_{\kappa} v_t-v_s$ 
for $0\leq s\leq t\leq T$.
One can show, see Corollary 4.2 in \cite{K2022b}, that 
for any smooth functions $u$, $v$ on $\bR^d$, and for  
$p, q\geq0$ and $\kappa\geq d+1+(p+q)d$, 
$$
|Du|^p\leq N\int_{\bR^d}\frh(\eta)|u_{(\eta)}|^p\,d\eta
$$
\begin{equation}
                                                                                    \label{1.14.6.24}
|Du|^p|D^2v|^q
\leq N\int_{\bR^d}\frh(\eta)|u_{(\eta)}|^p|Dv_{(\eta)}|^q\,d\eta
\end{equation}
holds, where $u_{(\eta)}:=\eta^iD_iu$ for functions $u$ on $\bR^d$. 
As in \cite{K2022b}, 
we introduce 
$$
v_t=v_t(x,\eta)=u_{t(\eta) }(x), 
\quad\theta_t=\sign v_t,
$$
$$
U_t=|u_t|^{p/2},
\quad 
V_t=|v_t|^{p/2}, 
\quad 
W_t=|\eta|^{p/2}|Du_t|^{p/2-1}|D^2u_t|. 
$$
Observe that 
\begin{equation}                                                       
                                                                                   \label{6.27.3.24}
|DV_t|^2\leq NW_t^2
\end{equation}
with $N=N(d,p)$, 
and that Corollary 4.2 in \cite{K2022b} implies 
\begin{equation}                                                 
                                                                                   \label{7.27.3.24}                                
|\eta|^p|Du|^p\prec_{\kappa}N|v_t|^p, 
\quad W_t^2\prec_{\kappa} N|DV_t|^2 
\end{equation}
with $N=N(d,p,\kappa)$. Substituting $-\varphi_{(\eta)}$ in place of 
$\varphi$ in \eqref{solution} and taking into account 
$$
-(D_i(a^{ij}D_ju+\beta^iu,\varphi_{(\eta)})
=(a^{ij}D_ju,\partial_{\eta}D_i\varphi)
$$
$$
=-(a^{ij}_{(\eta)}D_ju+a^{ij}D_ju_{(\eta)},D_i\varphi)
$$
$$
=(D_i(a^{ij}_{(\eta)}D_ju+a^{ij}D_ju_{(\eta)}),\varphi), 
$$
for $v=u_{(\eta)}$ we get 
$$
dv_t=\big(D_i(a_t^{ij}D_jv_t+a^{ij}_{(\eta)}D_ju_{t})
+ (b^i_tD_iu_t+c_tu_t+f_t)_{(\eta)}+g_{t(\eta) }\big)\,dt
$$
$$
+\big(M^k_tv_t+\sigma^{ik}_{t(\eta)}D_iu_t
+\nu^k_{t(\eta)}u_t+h^k_{t(\eta)}\big)\,dw_t^k. 
$$
Note that differently from the corresponding expression in 
\cite{K2022b} we keep the term 
$ (b^i_tD_iu_t+c_tu_t+f_t+g_t)_{(\eta)}$ 
in divergence form, to avoid additional regularity conditions 
on $b$, $c$, $f$ for the estimate \eqref{estimate Dp} to hold.  
By It\^o's 
formula (see Theorem 2.1  in \cite{K_10}), 
$$
(1/p)|v_t|^p_{L_p}
=                                                                                          
-(p-1)\Big(
\int_{\bR^d}|v_t|^{p-2}D_iv_t(a^{ij}_tD_jv_t
+a^{ij}_{t(\eta)}D_ju_{t})\,dx
\Big)\,dt 
$$
\begin{equation*}                    
-(p-1)
\Big(
\int_{\bR^d}(\theta_{t}|v_t|^{p-2} u_{t(\eta)(\eta)} (b^i_tD_iu_t+c_tu_t+f_t+g_t)\,dx
\Big)\,dt
\end{equation*}
\begin{equation}
                                                                             \label{5.25.3.24}
+\Big(\int_{\bR^d}J_t\,dx\Big)\,dt+m_t, 
\end{equation} 
where 
$$
J_t=((p-1)/2)|v_t|^{p-2}\sum_{k}(M_t^kv_t+\sigma^{ik}_{t(\eta)}D_iu_t
+\nu_{t(\eta)}^ku_t+h_{t(\eta)}^k)^2
$$
$$
\leq ((p-1)/2)|v_t|^{p-2}\sum_{k}\big(\sigma_t^{ik}D_iv_t\big)^2+ 
\delta/(2p)|DV_t|^2
$$
\begin{equation}                                                        \label{5.23.3.24} 
+N|\nu _t|^2|v_t|^{p}+N|\eta|^2|v_t|^{p-2}
(|D\sigma_t|^2Du_t|^2+|D\nu _t|^2|u_t|^2+|Dh _t|^2) 
\end{equation}
with a constant $N=N(d,p,\delta)$, and 
$$
m_t=\int_0^t\int_{\bR^{d}}J_s^k\,dx\, dw^k_s
$$
with 
$$
J^k_s=\theta_s|v_s|^{p-1}(M_s^kv_s+\sigma^{ik}_{s(\eta)}D_iu_s
+\nu_{s(\eta)}^k+h_{s(\eta)}^k). 
$$ 
Below we collect the integrands in \eqref{5.25.3.24} 
and \eqref{5.23.3.24},  
arrange them in suitable expressions $I_1$,...,$I_8$ 
to estimate them and their integrals separately. 
Above we have dropped the space argument,
now we also drop for some time the time argument for simplicity.
Collecting the terms containing $a^{ij}$, $\sigma^{ik}$ and $|DV_t|^2$
we have
$$
I_1:=-(p-1)|v|^{p-2}a^{ij}D_ivD_jv+\tfrac{p-1}{2}
|v|^{p-2}\sum_k(\sigma^{ik}D_iv)^2+(\delta/2p)|DV|^2
$$
$$
=-\tfrac{p-1}{2}(2a^{ij}-\sigma^{ik}\sigma^{jk})
|v|^{p-2}D_ivD_jv+\delta/(2p)|DV|^2
$$
$$
=-\tfrac{2(p-1)}{p^2}
(2a^{ij}-\sigma^{ik}\sigma^{jk})D_iVD_jV+(\delta/2p)|DV|^2. 
$$
Thus by Assumption \ref{assumption 1} and by \eqref{7.27.3.24}
$$
I_1\leq -\tfrac{2(p-1)}{p^2}\delta |DV|^2+(\delta/2p)|DV|^2
$$
\begin{equation}  
                                                                         \label{7.23.3.24}
\leq -(\delta/2p)|DV|^2\prec_{\kappa}-\hat\delta(|DV|^2+W^2)
\end{equation}
with a constant $\hat\delta=\hat\delta(d,p,\delta)>0$. 

In what follows, without losing generality, we assume
$$ 
\widehat{Da}+\widehat{D\sigma}+\hat b+\hat{c}+\hat{\nu}
+\widehat{D\nu}\leq1. 
$$
By Young's inequality, by the 
first estimate in \eqref{7.27.3.24} 
and by Lemma \ref{lemma 12.17.1}
we get 
$$
I_2:=-(p-1)|v|^{p-2}a^{ij}_{(\eta)}D_ivD_ju
\leq (p-1)|v|^{p-2}|\eta||D^2u||Da||\eta||Du|
$$
$$
\leq |v|^{p-2}((\hat\delta/4)|\eta|^2|D^2u|^2+N|Da|^2|\eta|^2|Du|^2)
$$
$$
=(\hat\delta/4)W^2+N|Da|^2|v|^{p-2}|\eta|^2|Du|^2
\leq (\hat\delta/4)W^2+N|Da|^2|\eta|^p|Du|^p
$$
$$
\prec_{\kappa} (\hat\delta/4)W^2+N|Da|^2V^2 
$$
\begin{equation*}
                                                                             \label{4.25.3.24}
\prec_{\kappa}(\hat\delta/4)W^2+N_1
\widehat{Da} |DV|^2
+N(1+\overline{Da}^2)V^2
\end{equation*}
with constants $N_1=N_1(d,p,r,\delta)$ and 
$N=N(d,p,r,\delta,\rho_0)$. 
We get in the same way  
$$
I_3:=-(p-1)\theta |v|^{p-2}u_{(\eta)(\eta)}b^iD_iu
$$
\begin{equation*} 
                                                                    \label{3.24.3.24}
\prec_{\kappa} (\hat\delta/8)W^2
+N_2\hat b|DV|^{2}+N(1+{\bar b}^2)V^2
\end{equation*}
with constants $N_{2}
=N_{2}(d,p,r,\delta)$ and 
$N=N(d,p,r,\delta,\rho_0)$. 
Similarly,  
$$
I_4:=-(p-1)\theta|v|^{p-2}u_{(\eta)(\eta)}cu
\leq (p-1)|v|^{p-2}|\eta||D^2u||\eta||cu|
$$
$$
\leq |v|^{p-2}\big((\hat\delta/16) |\eta|^2|D^2u|^2
+Nc^2|\eta u|^2\big)
=(\hat\delta/16) W^2+Nc^2|v|^{p-2}|\eta|^2u^2
$$
$$
\leq(\hat\delta/16) W^2+Nc^2|v|^{p}+Nc^2|\eta|^p|u|^p
\leq(\hat\delta/16) W^2+Nc^2V^{2}+N{c}^2|\eta|^pU^2 
$$
with a constant $N=N(d,p,\delta)$.  
 Hence by Lemma \ref{lemma 12.17.1} and 
by the first estimate in \eqref{7.27.3.24} 
(remember also that $\hat c\leq 1$), we have 
\begin{equation*}
                                                                                \label{4.24.3.24}
I_4\prec_{\kappa} (\hat\delta/16) W^2
+N_3{\hat c}|DV|^{2}+
N{\hat c}|\eta|^p|DU|^2 
+N(1+{\bar c}^2)(V^2+|\eta|^pU^2) 
\end{equation*}
$$
\prec_{\kappa}(\hat\delta/16) W^2
+N_3{\hat c}|DV|^{2}
+N(1+{\bar c}^2)(V^2+|\eta|^pU^2) 
$$
with $N_3=N_3(d,p,\delta,r)$ 
and $N=N(d,p,\delta,r,\rho_0)$.
In the same pattern we  get 
$$
I_5:=N|\eta|^2|v|^{p-2}|D\sigma|^2|Du|^2
$$
$$
\leq N((p-2)/p)|D\sigma|^2V^2+N(2/p)|D\sigma|^2|\eta|^p(|Du|^{p/2})^2
$$
\begin{equation*}
                                                                                             \label{5.24.3.24}
\prec_{\kappa}N_4\widehat{D\sigma}
(W^2+|DV|^2)
+N'(1+\overline{D\sigma}^2)V^2,
\end{equation*}
\begin{equation}
                            \label{1.25.3.24}
I_6:=N|\nu|^2|v|^{p}
\prec_{\kappa} N_5\hat\nu|DV|^2
+N'(1 +\bar\nu^2)V^2 ,
\end{equation}
$$
I_7=N|\eta|^2|v|^{p-2}|D\nu|^2|u|^2
$$
$$
\leq (p-2)/p)N|D\nu|^2V^2+(2/p)N|D\nu|^2|\eta|^pU^2
$$
\begin{equation}
                                      \label{2.25.3.24}
\prec_{\kappa} N_6\widehat{D\nu}|DV|^2
+N'(1+\overline{D\nu}^2)(V^2+|\eta|^pU^2), 
 \end{equation} 
with the constant  $N$ in \eqref{5.23.3.24}, and constants 
$N'=N'(d,p,\delta,r,\rho_0)$ and $N_{i}=N_i(p,d,\delta,r)$ ($i=4,5,6$). 
For the terms containing $f$, $g$ and $Dh$ we have
$$
I_8:=-(p-1)\theta |v|^{p-2}u_{(\eta)(\eta)}f
+\theta |v|^{p-1}g_{(\eta)} 
+N|\eta|^2|v |^{p-2}|Dh  |^2  
$$
$$
\leq(p-1)W  |v|^{p/2-1}|\eta|\,|f|
+|\eta||v|^{p-1}|Dg|+N|\eta|^2|v |^{p-2}
|Dh  |^2  
$$
\begin{equation} 
                                                                            \label{3.25.3.24}
\leq(\hat\delta/16)W^{2}+N'I 
\end{equation} 
with the constant $N$ in \eqref{5.23.3.24}, 
a constant $N'=N'(d,p,\delta)$, and  
$$
I_{t}:=|v_{t}|^{p-2}|\eta|^2|f_{t}|^2+
|\eta||v_{t}|^{p-1}|Dg_{t}|+|v_{t}|^{p-2}
|\eta|^2|Dh_{t}|^2 . 
$$
Thus, subjecting   
$\widehat{Da}$, $\hat b$, $\hat c$, 
$\hat \nu$ and $\widehat{D\nu}$ to 
$$
N_1\widehat{Da} \leq\hat\delta/4, 
\quad
N_2{\hat b} \leq\hat\delta/8, 
\quad
N_3{\hat c} \leq\hat\delta/16, 
\quad
N_4\widehat{D\sigma} \leq\hat\delta/32
$$
\begin{equation}            
                                                                          \label{1.23.9.24}
N_5{\hat \nu} \leq\hat\delta/64, 
\quad                                     
N_6\widehat{D\nu} \leq\hat\delta/64,
\end{equation}
and taking into account \eqref{7.23.3.24} through 
\eqref{3.25.3.24}, 
from equation \eqref{5.25.3.24} we obtain 
$$
d|v_t|^p\prec_{\kappa} p \sum_{i=1}^{8}I_i(t)\,dt
+ p J^k_t\,dw^k_t
$$
$$
\prec_{\kappa}- 
 \hat\delta (|DV_t|^2+W^{2}_t )\,dt
+ N'\gamma'_t |v_t|^p\,dt+ N(d,p,\delta) I_t\,dt
$$
\begin{equation*} 
                                                                         \label{7.25.3.24}
+ N' (1+\overline{D\nu}^2 +\bar c_t^2 )|\eta|^pU_t^2\,dt
+ p J^k_t\,dw^k_t, 
\end{equation*}
 with $\gamma'=
1+\overline{Da}^2+\bar b^2+\bar c^2+
\overline{D\sigma}^2+\bar\nu^2$ 
and constants $N'=N'(d,p,\delta,r,\rho_0)$,
which  for $C\geq N'$  implies 
$$
d(e^{-p\Psi_t}|v_t|^p)\prec_{\kappa}
- \hat\delta e^{-p\Psi_t}(|DV_t|^2+W_t^2)\,dt
- \Lambda_te^{-p\Psi_t}|v_t|^p\,dt
$$
$$
+Ne^{-p\Psi_t}I_t\,dt+N'e^{-p\Psi_t}
(1+\overline{D\nu}^2 +\bar c^2 )|\eta|^pU_t^2\,dt
+ p e^{-p\Psi_t}J^k_t\,dw^k_t
$$
with $N=N(d,p,\delta)$ 
and $N'=N'(p,d,\delta,r,\rho_0)$. Converting this 
into integral form, with integrals with respect to $x$, $\eta$ and $t$ 
and using \eqref{1.14.6.24},  
we get that almost surely 
$$
e^{-p\Psi_t}|Du_t|^p_{L_{p}}
+\int_0^t\int_{\bR^d}e^{-p\Psi_s}(|Du_s|^{p-2}|D^2u_s|^2
+\Lambda_s|Du_s|^p)dx\,ds
$$
$$
\leq N|Du_0|_{L_p}^p
+ N' \int_0^te^{-p\Psi_s}
(1+\overline{D\nu_s}^2 +\bar c_s^2 )|u_s|^p_{L_p})\,ds
$$
 \begin{equation}                                                              
                                                                              \label{1.27.3.24}
+N\int_0^te^{-p\Psi_s}\int_{\bR^{2d}}\frh(\eta)I_s\,dxd\eta\,ds
+N\int_0^te^{-p\Psi_s}\int_{\bR^{2d}} \frh(\eta)J^k_s\,dx d\eta\,dw^k_s 
\end{equation}
for $t\in[0,T]$, with constants $N=N(d,p,\delta)$ 
and $N'=N'(d,p,\delta,r,\rho_0)$. Hence, 
 taking $\lambda\geq C\geq N'$ and 
$
\mu\geq N'\overline{D\nu}^2   
$
in \eqref{1.22.9.24}, by Theorem \ref{theorem p}   
for any stopping time $\tau\leq T$  we have  
$$
E\int_0^{\tau}e^{-p\Psi_t}\int_{\bR^d}
(|Du_t|^{p-2}|D^2u_t|^2+\Lambda_t|Du_t|^p)\,dx\,dt
$$ 
\begin{equation}
                                                                               \label{1.26.3.24}
\leq NE|u_0|^p_{W^1_p}
+NE\int_0^{\tau}\int_{\bR^{2d}}e^{-p\Psi_t}\frh(\eta)I_t\,dxd\eta\,dt
+NE\frK( T ) 
\end{equation}
 with a constant $N=N(d,p,\delta)$, where  
$$
\frK( T ):=\Big(\int_0^{ T }| f_{t}e^{-\psi_{t} }|^{2}_{L_{p}}
+|h _t e^{-\psi_{s}}|^{2}_{L_{p} }\,dt\Big)^{p/2}
+\Big(\int_{0}^{ T }  
|g_{t} e^{- \psi_{t}}| _{L_{p}}
\,dt\Big)^{p}.   
$$ 
In addition,
by taking into account  \eqref{4.24.5.24}, we get
$$
\bG(\tau):=E\sup_{t\leq \tau}|e^{-p\Psi_t}u_t|^p_{W^1_p}
$$
$$
+E\int_0^{\tau}\int_{\bR^d}e^{-p\Psi_s}(|Du_s|^{p-2}|D^2u_s|^2
+\Lambda_s|Du_s|^p)dx\,ds
$$
$$
\leq N E|u_0|_{W^1_p}^p
+NE\frK( T )
$$
 \begin{equation}                                                             
                                                                         \label{1.27.3.24a}
+NE\int_0^{\tau}e^{-p\Psi_s}\int_{\bR^{2d}}\frh(\eta)I_s\,dxd\eta\,ds
+NE\frJ(\tau)
\end{equation}
 with $N=N(d,p,\delta)$,  where
$$
\frJ(\tau):=\sup_{t\leq\tau} \Big|\int_0^{t}e^{-p\Psi_s}\int_{\bR^{2d}}
 \frh(\eta)J^k_s\,dx d\eta\,dw^k_s\Big|. 
$$
By the Davis inequality 
$$
E\frJ(\tau)
\leq 3E\Big(\int_0^{\tau} e^{-2p\Psi_s}\sum_{k}\Big|\int_{\bR^{2d}}
\frh(\eta)J^k_s\,dxd\eta\Big|^2 ds\Big)^{1/2}.
$$
Noting that 
$$
J^k_s=\theta_s|v_s|^{p-1}(M_s^kv_s+\sigma^{ik}_{s(\eta)}D_iu_s
+\nu_{s(\eta)}^k+h_{s(\eta)}^k)
$$
$$
\leq |v_s|^{p/2}\big(\,|v_s|^{p/2-1}
|M_s^kv_s+\sigma^{ik}_{s(\eta)}D_iu_s
+\nu_{s(\eta)}^k+h_{s(\eta)}^k|\big), 
$$
by Cauchy-Schwarz-Bunjakovsky we get 
$$
\sum_{k}\Big|\int_{\bR^{2d}}
 \frh(\eta)J^k_s\,dxd\eta\Big|^2
 \leq N|v_s|_{L_p}^p\int_{\bR^{2d}}\frh(\eta)J_s\,dxd\eta.
$$
Thus, with the constant $N$ from  \eqref{1.27.3.24a}, we have 
\begin{align}
NE\frJ(\tau)\leq & 
(1/4)E\sup_{t\leq\tau}|e^{-p\Psi_s}_tu_{t}|^p_{W^1_p}                     \nonumber\\ 
&+NE\int_0^{\tau}e^{-p\Psi_s}\int_{\bR^{2d}}\frh(\eta)J_s\,dxd\eta\,ds  \label{1.24.9.24}
\end{align}
with a constant 
$N=N(d,p, \delta)$.  
By \eqref{1.25.3.24}, \eqref{2.25.3.24} 
and \eqref{3.25.3.24}, 
taking into account  \eqref{6.27.3.24}-\eqref{7.27.3.24}, 
from \eqref{5.23.3.24} we get 
$$
J_s\prec_{\kappa}|\eta|^p\big(N|Du_s|^{p-2}|D^2u_s|^2
+ N'\gamma''_s |Du_s|^p
+ N'(1+\overline{D\nu_s}^2) |u_s|^p\big)+NI_s,  
$$
 with 
$\gamma''=1+\bar\nu^2+\overline{D\sigma}^2+\overline{D\nu}^2$, constant  
$N'=N'(d,p,\delta,r,\rho_0)$ and constant $N$, 
which due to \eqref{1.23.9.24} depends only on $d,p,\delta$.
Consequently, 
$$
E\int_0^{\tau }e^{-p\Psi_s}\int_{\bR^{2d}}\frh(\eta)J_s\,dxd\eta\,ds
$$
$$
\leq
 E\int_0^{\tau}e^{-p\Psi_s}\int_{\bR^{d}}
( N |Du_s|^{p-2}|D^2u_s|^2
+ N'\gamma''_s |Du_s|^p
+ N'(1+\overline{D\nu_s}^2) |u_s|^p)dx\,ds
$$
$$
+NE\int_0^{\tau}\int_{\bR^{2d}}e^{-p\Psi_t}\frh(\eta)I_t\,dxd\eta\,dt  
$$
 with constants $N=N(d,p,\delta)$ 
and $N'=N'(d,p,\delta,r,\rho_0)$. 
Hence, taking $C\geq N'$ in the definition of $\Lambda$, and  taking 
$\lambda\geq C\geq N'$ and $\mu\geq N'\overline{D\nu}^2$ 
in the definition of $\alpha$ 
in \eqref{1.22.9.24}, by virtue of \eqref{1.26.3.24} 
and Theorem \eqref{theorem p} 
we get  
$$
E\int_0^{\tau }e^{-p\Psi_s}\int_{\bR^{2d}}\frh(\eta)J_s\,dxd\eta\,ds
$$
\begin{equation}      
                                                                      \label{2.24.9.24}
\leq NE|u_0|^p_{W^1_p}
+N
E\int_0^{\tau}\int_{\bR^{2d}}e^{-p\Psi_t}\frh(\eta)I_t\,dxd\eta\,dt                                                  
+NE\frK(T).
\end{equation}
 
Thus from \eqref{1.27.3.24a}, taking into account 
\eqref{1.24.9.24} and \eqref{2.24.9.24} we obtain 
$$
\bG(\tau)\leq NE|Du_0|^p_{L_p}
+(1/4)E\sup_{t\leq\tau }|e^{-\Psi_s}u_t|^p_{L_p}
+NE\frK( T  )
$$
\begin{equation}
                                                                           \label{9.27.3.24}
+NE\int_0^{\tau}\int_{\bR^{2d}}e^{-p\Psi_t}\frh(\eta)I_t\,dxd\eta\,dt                                                  
\end{equation}
with a constant $N=N(d,p, \delta)$. 
Here for the last term we have 
$$
NE\int_0^{\tau}\int_{\bR^{2d}}e^{-p\Psi_t}\frh(\eta)I_t\,dxd\eta\,dt
$$
$$
\leq N'E\int_0^{\tau}
e^{-p\Psi_s}
|v_s|^{p-2}_{L_p}|f_s|^2_{L_p}
+|v_s|_{l_p}^{p-1}|Dg_s|_{L_p}+|v_s|^{p-2}_{L_p}
|Dh_s|^2_{L_p}\,ds 
$$
$$
+(1/4)E\sup_{t\leq\tau}|e^{-\Psi_t} v_s |^{p}_{L_p}
+N''E
\Big(
\int_0^{\tau}(|e^{-\Psi_s}f_s|^2_{L_p}
+|e^{-\Psi_t}Dh_s|_{L_p}^2)\,ds
\Big)^{p/2}
$$
$$
+N''\Big(\int_0^{\tau}|e^{-\Psi_s}Dg_s|_{L_p}\,ds\Big)^p 
$$
 with constants $N'$ and $N''$ depending 
only on $d$, $p$ and $\delta$.  Consequently, from \eqref{9.27.3.24} we obtain 
$$
\bG(\tau)\leq (1/2)E\sup_{t\leq\tau}|e^{-\psi_t}|u_{t}|^p_{W^1_p}
+NE|u_0|^p_{W^1_p}+N\bK( T )
$$
with a constant $N=N(d,p, \delta)$, 
where 
$$
\bK( T )=E\Big(\int_0^{ T }
|e^{-\Psi_s}g_s|_{W^1_p}\,ds
\Big)^{p}
+E\Big(
\int_0^{ T }
(|e^{-\Psi_t}f_s|^2_{L_p}+|e^{-\Psi_s}h_s|_{W^1_p}^2)\,ds
\Big)^{p/2}. 
$$
Taking here 
$$
\tau_n=\inf\{t\in[0,T]:|u_t|_{W^2_p}\geq n\}
$$
in place of $\tau$,  we get estimate \eqref{estimate Dp} with $\tau_n$ 
in place of $T$, which implies \eqref{estimate Dp} 
by Fatou's lemma as $n\to\infty$. 

Now we dispense with the additional assumptions by approximating 
\eqref{eq1}--\eqref{eq2} by \eqref{eqe}, 
defined in the proof of Theorem \ref{theorem p}. Then due to 
what we have proved above, for the solution $u^{\varepsilon}$ 
of we have \eqref{eqe} we have  
$$
E\sup_{t\leq T}|e^{-\Psi_t}u^{\varepsilon}_t|^p_{W^1_p}
+E\int_0^{T}
e^{-p\Psi_t}(\big||Du^{\varepsilon}_t
|^{p/2-1} D^2u^{\varepsilon}_t \big|^2_{L_2}
+\Lambda_t|u^{\varepsilon}_t|^p_{W^1_p})\,dt$$
\begin{equation*}
                                                                   \label{estimate 3}
\leq NE|u_0|^p_{W^1_p}+
N\bK(T)
\end{equation*}
with a constant $N=N(d,p,r,\delta)$.

Naturally, we are going to apply Theorem \ref{theorem 4.22.1}
to prove that, for any sequence $\varepsilon_{n}\downarrow 0$,
\begin{equation}
                                                                            \label{7.9.6.240}
\lim_{n\to\infty}E\sup_{t\in[0,T]}
|u^{\varepsilon_n}_t-u_t|^2_{W^{1}_{2}}
+\lim_{n\to\infty}
E \int_0^T|D^{2}u^{\varepsilon_n}_t-D^{2}u_t|^2_{L_2}\,dt=0.                                                                  
\end{equation}
 To this end we set 
$H=W^1_2$, $V=W^2_2$, and cast \eqref{eq1}--\eqref{eq2} 
into the stochastic evolution equation \eqref{4.28.1} as in the proof 
of Theorem \ref{theorem 2.5.24}, 
see \eqref{6.16.6.24} through \eqref{10.16.6.24}. 
In the same way, we cast \eqref{eqe} 
with $\varepsilon=\varepsilon_n$ for 
a sequence $\varepsilon_n
\to0$ into  the stochastic 
evolution equations  
$$
dv^n_{t}=\big[ A^n_{t}u_{t} +\sca^{n\ast}_{t}v^n_{t}
+\scc^n_{t}v^n_{t}
+f^{*n}_{t} + f^n_{t}+g^n_{t} \big]\,dt
$$
\begin{equation*}
+\big( B^{n k}_{t}v^n_{t}+\scb^{n k}_{t}v^n_{t} 
+h^{n k}_{t}\big) \,dw^{k}_{t},\quad
t\leq T,\quad v^n_{t}\big|_{t=0}=u_{0}^{(\varepsilon_n)}
\end{equation*}
for integers $n\geq1$. Then we check the conditions
of Theorem \ref{theorem 4.22.1}. 
Using well-known properties of mollifications 
and taking into account \eqref{1.1.4.24} we can easily check that 
parts (i), (ii) and (iii) of Assumption \ref{assumption 4.22.1} hold. 
To verify part (iv) of this assumption,  first we deal with
\eqref{5.10.6.24} and
let $v\in C_0^{\infty}$ be vanishing outside of a ball $B_R$ 
of radius $R$. Then for the unit ball $B_V$ in $V=W^2_2$
$$
\sup_{\phi\in B_V}
\big|\big((1-\Delta)\phi,(b^{Mi}_{t}
-b^{M(\varepsilon)i})D_{i}v\big)\big| 
\leq  \sup_{\bR^d}|Dv|
\big|{\bf1}_{B_R}(b^{M }_t-b_t^{M(\varepsilon) })\big|_{L_{2}}, 
$$
$$
\sup_{\phi\in B_V}
\big|\big((1-\Delta)\phi,(c^{M}_{t}
-c^{M(\varepsilon) }_{t}) v\big)\big| 
\leq \sup_{\bR^d}|v|
\big|{\bf1}_{B_R}(c^{M}_{t}
-c^{M(\varepsilon) }_{t})\big|_{L_{2}}, 
$$
$$
\sup_{\phi\in B_V}
\big|(D_i\phi, (a^{ij}_t-a^{(\varepsilon)ij}_t)D_jv )
+(D_kD_i\phi,((D_ka)^{Mij}_t-
(D_ka)^{M(\varepsilon)ij}_t)D_jv)
$$
$$
+(D_kD_i \phi,
(a^{ij}_t-a^{(\varepsilon)ij}_t)D_kD_jv)\big|
$$
$$
\leq
N\sup_{\bR^{d}}(|Dv|+|D^{2}v|)\big(
\big|{\bf1}_{B_R}(a_{t}
-a^{(\varepsilon) }_{t})\big|_{L_{2}}
+\big|{\bf1}_{B_R}((Da)^{M}_{t}
-(Da)^{M(\varepsilon) }_{t})\big|_{L_{2}}\big).
$$
This yields that for any $t$ we have
$| (A^{n}_{t}-A _{t})v 
 | ^{2}_{V^{*}}\to0$ as $n\to\infty$.
 Quite similarly one proves that
 $| (B^{n}_{t}-B _{t})v  | ^{2}_{\ell_{2}(H)}\to0$ as
  $n\to\infty$ and this gives us \eqref{5.10.6.24}.

To check \eqref{4.23.1} observe that
$$
|\sca^{*}_{t}v-\sca^{n*}_{t}v|^{2}_{V^{*}}
\leq N\big(|\,|b_{t}^{B}-b_{t}^{B\varepsilon_{n}}|Dv|^{2}_{L_{2}}+
|\,|c_{t}^{B}-c_{t}^{B\varepsilon_{n}}|v|^{2}_{L_{2}}
$$ 
$$
+|\,|(D a)^{B }_t-(D a)^{B\varepsilon_{n}}_t|Dv|^{2}_{L_{2}}.
$$
This implies the second equality in \eqref{4.23.1}
in the same way as after \eqref{7.10.6.24}. 
Similarly the remaining relations in \eqref{4.23.1}
are checked. This leads to \eqref{7.9.6.240}.

After that it only remains to reproduce
the end of the proof of Theorem \ref{theorem p}
with obvious changes. The theorem is proved.

\begin{remark}
By an appropriate (minor) change of the above proof we can get   
that the above theorem remains valid if in Assumption \ref{assumption Dp8}
we suppose that $c$ and $Dc$ are 1/2-admissible, instead of assuming 
that $c$ is admissible. 
\end{remark}

\end{document}